\documentclass[12pt, reqno]{amsart}
\usepackage{amsmath,amssymb,amsthm}

\usepackage{amsmath, amssymb}
\usepackage{amsthm, amsfonts, mathrsfs}
\usepackage{mathptmx}
\usepackage{fullpage}
\usepackage{amsfonts,graphicx}
\numberwithin{equation}{section}
\usepackage[colorlinks=true, pdfstartview=FitV, linkcolor=blue, citecolor=blue, urlcolor=blue]{hyperref}

\newcommand{\ZZ}{\mathbb{Z}}

\newcommand{\NN}{\mathbb{N}}
\newcommand{\Ss}{\mathbb{S}}

\newcommand{\RR}{\mathbb{R}}

\newcommand{\EE}{\varepsilon}

\newcommand{\DD}{\textnormal{D}}

\newcommand{\Div}{\textnormal{div}}

\newcommand{\supp}{\textnormal{supp }}
\newcommand{\Lip}{\textnormal{Lip}}
\newcommand{\loc}{\textnormal{loc}}
\newcommand{\Curl}{\textnormal{curl}}

\newcommand{\rot}{\textnormal{rot}}

\newtheorem{Theo}{Theorem}[section]
\newtheorem{lem}[Theo]{Lemma}
\newtheorem{cor}[Theo]{Corollary}
\newtheorem{prop}[Theo]{Proposition}

\theoremstyle{plain}
\theoremstyle{definition}
\newtheorem{defi}[Theo]{Definition}

\theoremstyle{remark}
\newtheorem{Rema}[Theo]{Remark}
\newtheorem*{rema*}{Remark}
\newtheorem*{remas}{Remarks}

\author[Y. Maafa]{Youssouf Maafa}
\address{LEDPA, Universit\'e de Batna --2--\\ Facult\'e des Math\'ematiques et Informatique\\ D\'epartement de Math\'ematiques\\ 05000 Batna Alg\'erie}
\email{y.maafa@univ-batna2.dz}

\author[M. Zerguine]{Mohamed Zerguine}
\address{LEDPA, Universit\'e de Batna --2--\\ Facult\'e des Math\'ematiques et Informatique\\ D\'epartement de Math\'ematiques\\ 05000 Batna Alg\'erie}
\email{m.zerguine@univ-batna2.dz}

\date{}

\begin{document}

\title[Inviscid limit]
{Inviscid limit  for the viscous 2D Boussinesq system with temperature-dependent diffusivity}

\maketitle
\begin{abstract} 
We establish global-posedness in time for the viscous Boussinesq equations in two dimensions of space with temperature-dependent diffusivity in the framework of a smooth vortex patch. We also provide the inviscid limit for velocity, temperature, and associated flow toward the system studied very recently in \cite{Paicu-Zhu} as soon as the viscosity goes to zero, and quantify the rate of convergence.
\end{abstract}

\tableofcontents

\section{Introduction} 
\subsection{Model and synopsis of results} In the modeling of geophysical fluid dynamics, the associated fields of pressure, temperature, stratification, and density,\ldots, are necessary factors since they are dynamically linked to the motion of the oceans and atmosphere. Generally, it is very hard to explore these dynamics either theoretically or experimentally. For theoretical fluid mechanics the difficulty springs fundamentally from the overlap of the aforementioned factors. Nevertheless, simplification and approximation guided to the brilliant of set equations, one of them Boussinesq type systems, see \cite{Pedlosky}. In the derivatives of such systems, it is usual to assume that the fluid viscosity and thermal conductivity are positive constants; however, there are several important physical situations where such hypotheses are not adequate, and one must consider the possibility that such viscosity and thermal conductivity may be temperature-dependent, see, e.g. \cite{Boldrini-Climent-Ezquerra-Rojas-Medar2}. The Cauchy problem associates with such phenomena are given by the following set equations. 
\begin{equation}\label{Eq-1}
\left\{ \begin{array}{ll} 
\partial_{t}v+v\cdot\nabla v-\nabla\cdot\big(\mu(\theta)\nabla v\big)+\nabla p=G(\theta) & \textrm{if $(t,x)\in \RR_+\times\RR^2$,}\\
\partial_{t}\theta+v\cdot\nabla \theta-\nabla\cdot\big(\kappa(\theta)\nabla \theta\big) =0 & \textrm{if $(t,x)\in \RR_+\times\RR^2$,}\\ 
\nabla\cdot v =0, &\\ 
(v,\theta)_{| t=0}=(v_0,\theta_0).\tag{GB}
\end{array} \right .
\end{equation}
Usually, $v(t,x)\in\RR^2$ refers to the velocity vector field localized in variable space $x\in\RR^2$ at a time variable $t$ which is assumed to be incompressible $\nabla\cdot v=0$, and $p(t,x)\in\RR$ and $\theta(t,x)\in\RR^{\star}_{+}$ are a thermodynamical variables representing respectively the pressure and the temperature. The buoyancy effects on the fluid expresses by the following vector-valued function$G(\theta)=G_1(\theta)\vec{e}_1+G_2(\theta)\vec{e}_2$, satisfying $G\in C^3(\RR^2)$ and $G(0)=(0,0)$, with $\vec e_1=(1,0)$ and $\vec e_2=(0,1)$. The scalar functions $\mu(\cdot)$ and $\kappa(\cdot)$ are smooth and fulfill with their derivatives the following bound condition
\begin{equation}\label{Eq-2}
\mu_0^{-1}\le \mu(\theta)\le\mu_0\;\mbox{and}\; \mu'(\theta)\le\mu_0,\quad \kappa_0^{-1}\le \kappa(\theta)\le\kappa_0\;\mbox{and}\; \kappa'(\theta)\le\kappa_0. 
\end{equation}
\hspace{0.5cm}The experiments done by von Tippelkirch \cite{Tippelkirch}, for instance, clearly confirm the influence of the viscosity dependent-temperature on the main macroscopic features of the flow, and thus the necessity of analyzing such more complex situations. 

\hspace{0.5cm}Also remarkable that equations appear in the system \eqref{Eq-1} are strongly nonlinear compared to the classical Boussinesq, so we find some technical difficulties in dealing with them. 

\hspace{0.5cm}Let us recall some significative results, where $\mu(\theta)=\mu, \kappa(\theta)=\kappa$ are positive constant and $G(\theta)=(\theta,0)$. The generalized Boussinesq system \eqref{Eq-1} closes with the classical one which has widely studied whether theoretically or experimentally, especially the local/global well-posedness topic has received great attention in PDEs community. We embark by the work of Cannon and Dibenedetto \cite{Cannon-Dibenedetto} and Guo \cite{Guo} were exploited the classical method to gain the global regularity in $\RR^2$, while the case $\mu>0$ and $\kappa=0$ was successfully treated independently by Chae \cite{Chae} and Hou and Li \cite{Hou-Li} in the setting of subcritical Sobolev regularity which enhanced later by Abidi and Hmidi \cite{Abidi-Hmidi}, once $(v_0,\theta_0) \in B^{-1}_{\infty,1}\cap L^2\times B^{0}_{2,1}$. The opposite case $\mu=0$ and $\kappa>0$ is also well-investigated by Chae under the same regularity. Thereafter, the same result was extended by Hmidi and Keraani in \cite{Hmidi-Keraani-1} for critical Besov spaces $(v^0,\rho^0)\in B^{1+2/p}_{p,1}\times B^{2/p-1}_{p,1}\cap L^r,\;r>2$. For the Yudovitch's solutions the successful attempt goes back to Danchin and Paicu where they proved in \cite{Danchin-Paicu} that we can go beyond the strong solutions and establish the global existence and uniqueness for weak initial data in the weak sense. In the same way, Hmidi and the second author established in \cite{Hmidi-Zerguine0} a global well-posedness topic with fractional dissipation $(-\Delta)^{\frac{\gamma}{2}}$, with $\gamma\in]1,2]$ by exploring the Lagrangian variables. The critical case $\gamma=1$ was solved later in \cite{Hmidi-Keraani-Rousset-1} by using the special structure of the equations. For the other improvement and connected topics we refer the reader to \cite{Hmidi-Keraani-Rousset-2, Hmidi-Rousset, Larios-Lunasin-Titi, Miao-Xue,Tao-Zhang,Weinan-Shu,Wu-Xu,Wu-Xu-Xue-Ye}.

\hspace{0.5cm}In the general setting, that is $\mu(\cdot)$ and $\kappa(\cdot)$ are a functions fulfilling the requirement \eqref{Eq-2} and $G(\cdot)=(G_1(\cdot),G_2(\cdot))$, the situation becomes difficult. Worth mentioning, Lorca and Boldrini succeed to recover \eqref{Eq-1} locally in time for a strong solution for general initial data, whilst globally in time under smallness condition. Besides, they settled that the same system is globally well-posed in the context of a weak solution. Thereafter, Wang and Zhang established in \cite{Wang-Zhang} that \eqref{Eq-1} admits a unique global solution, once $v_0,\theta_0\in H^{s}$ with $s>2$. Showing crucially that $\theta$ is H\"older continuous by exploring De Giorgi argument and some background of harmonic analysis. This result was reached later by Sun and Zhang in \cite{Sun-Zhang} in a more general case, namely in a bounded domain for \eqref{Eq-1} and tridimensional infinite Prandtl number model with the viscosity and diffusivity depending on the temperature. In the same way, Francesco investigated in \cite{Francesco} that the system in question is globally well-posed for a weak solution in any dimension provided that initial temperature in only bounded and the initial velocity belongs to some critical Besov space concerning scaling invariance. In the situation where $\nabla\cdot\big(\kappa(\theta)\nabla\theta\big)$ is replaced by $\kappa(-\Delta)^{\frac12}\theta$, Abidi and Zhang proved in \cite{Abidi-Zhang} that the obtained system has a unique global solution whenever the viscosity closes to positive constant in $L^\infty-$norm. For large literature we refer to \cite{Li-Xu,Li-Pan-Zhang,Lorca-Boldrini-1,Lorca-Boldrini-2} and references therein.

\hspace{0.5cm}Our original motivation of this reasearch is to to reach the same result in the spirit of \cite{Paicu-Zhu} in the case, where $G(\theta_\mu)=\theta_\mu\vec e_2$, that is the following system.
\begin{equation}\label{Eq-3}
\left\{ \begin{array}{ll} 
\partial_{t}v_\mu+v_\mu\cdot\nabla v_\mu-\mu\Delta v_\mu+\nabla p=\theta_\mu\vec e_2 & \textrm{if $(t,x)\in \RR_+\times\RR^2$,}\\
 \partial_{t}\theta_\mu+v_\mu\cdot\nabla \theta_\mu-\nabla\cdot\big(\kappa(\theta_\mu)\nabla \theta_\mu\big) =0 & \textrm{if $(t,x)\in \RR_+\times\RR^2$,}\\ 
\nabla\cdot v_\mu =0, &\\ 
({v_\mu},{\theta_\mu})_{| t=0}=({v}_\mu^0,{\theta_\mu}^0).\tag{B$_\mu$} 
\end{array} \right .
\end{equation}
This study comprises twofold. The first one concerns the global well-posedness in time in the context of a smooth vortex patch. The second one addresses to the inviscid limit of the system \eqref{Eq-3} to inviscid one, 
\begin{equation}\label{Eq-4}
\left\{ \begin{array}{ll} 
\partial_{t}v+v\cdot\nabla v+\nabla p=\theta \vec e_2, &\\
 \partial_{t}\theta+v\cdot\nabla \theta-\nabla\cdot(\kappa(\theta)\nabla \theta) =0, &\\ 
\nabla\cdot v =0, &\\ 
(v,\theta)_{| t=0}=(v^0,\theta^0).\tag{B$_0$} 
\end{array} \right .
\end{equation}
whereas the associated flows when the viscosity goes to zero, and quantify the rate of convergence. 

\hspace{0.5cm}To formulate our problem it is very convenient to use the vorticity formulation for that equations. This quantity $\omega$ is very efficient in the analysis of fluid dynamics, in particular, it measures how fast the fluid rotates and can be identified as a scalar function $\omega=\partial_1 v^2-\partial_2 v^1$. To derive an evolution equation of $\omega$, taking the  $\Curl$ operator to the momentum equation in \eqref{Eq-6} one obtains
\begin{equation}\label{Eq-5}
\left\{ \begin{array}{ll}
\partial_{t}\omega_{\mu}+v_{\mu}\cdot\nabla\omega_{\mu}-\mu\Delta\omega_{\mu}=\partial_1\theta_{\mu}, & \\
\partial_{t}\theta_{\mu}+v_\mu\cdot\nabla \theta_{\mu}-\nabla\cdot(\kappa(\theta_{\mu})\nabla \theta_{\mu}) =0, & \\
v_{\mu}=\nabla^\perp\Delta^{-1}\omega_{\mu},\\
(\theta_{\mu},\omega_{\mu})_{| t=0}=(\theta_{\mu}^0,\omega_{\mu}^0).
\end{array} \right.\tag{VD$_{\mu} $}
\end{equation}
In the case, where the viscosity $\mu=0$, the previous system takes the form.
\begin{equation}\label{Eq-50}
\left\{ \begin{array}{ll}
\partial_{t}\omega+v\cdot\nabla\omega=\partial_1\theta, & \\
\partial_{t}\theta+v\cdot\nabla \theta-\nabla\cdot(\kappa(\theta)\nabla \theta) =0, & \\
v=\nabla^\perp\Delta^{-1}\omega,\\
(\theta,\omega)_{| t=0}=(\theta^0,\omega^0).
\end{array} \right. \tag{VD$_{0}$}
\end{equation}
Before telling the main results of this paper, let us retrieve some works regarding the smooth patch in a particular case, $\kappa(\theta)=\kappa$ is a positive constant. A vortex patch means that $\omega_{0}={\bf 1}_{\Omega_0}$, with $\Omega_0$, is a connected bounded domain. We recall that the vortex patch problem originated with the classical planar Euler equations, a particular case of the previous system, however, $\mu=0$ and $\theta=\theta^0$. So, Euler's equations take the form
\begin{equation}\label{Eq-6}
\left\{ \begin{array}{ll}
\partial_{t}\omega+v\cdot\nabla \omega=0, & \\
v=\nabla^\perp\Delta^{-1}\omega, &\\
\omega_{| t=0}=\omega^0.
\end{array} \right.\tag{E}
\end{equation}
Equations \eqref{Eq-6} are of nonlinear kind, consequently the characteristic method provides 
\begin{equation}\label{Eq-6}
\omega(t,x)=\omega_0(\Psi^{-1}(t,x)).
\end{equation}
with $\Psi$ is the flow generated by the velocity $v$,
\begin{equation}\label{Eq-7}
\left\{ \begin{array}{ll} 
\partial_{t}\Psi(t,x)=v(t,\Psi(t,x)), & \\
\Psi(0,x)=x.
\end{array} \right.
\end{equation}
In light of \eqref{Eq-6}, we discover that $\omega(t)={\bf 1}_{\Omega_t}$ with $\Omega_t=\Psi(t,\Omega_0)$ is also a vortex patch that moves through the time. It should be mentioned that the regularity of the time evolution domain $\Omega_t$ has gained a lot of attention, however, the successful attempt has been rigorously justified by Chemin \cite{Chemin2}  and subsequently developed in \cite{Bertozzi-Constantin,Serfati}. The Chemin's paradigm claims that the Lipschitz norm of velocity relies upon that of vorticity's stratied regularity $\partial_{X_t}\omega$ in H\"older space with a negative index $C^{\EE-1}$ with $0<\EE<1$ via a stationary logarithmic estimate with $X=(X_t)$ is a time dependent family of vector field characterized by a distinguished properties in our analysis, see subsection \ref{X_t}. The rest of the topics for the of Euler and Navier-Stokes system in different situations are accomplished in \cite{Danchin0,Depauw,Hmidi-1}.

\hspace{0.5cm}For the Boussinesq system \eqref{Eq-1}, the first result in this way is due to Hmidi and the second author, where they developed an elegant work in \cite{Hmidi-Zerguine} for $\mu=0, \nabla\cdot\big(\kappa(\theta)\nabla\theta\big)=\kappa\Delta$ and demonstrated that the obtained system is globally well-posed in time by exploring the asymptotic behavior of the density. The second author \cite{Zerguine}, reached the same result by replacing the full dissipation by critical one which is gaining a sharper assertion compared to Chemin's result about the classical Euler equations. Very recently for $\mu$ and $\kappa$ are positive constant the problem was done by Meddour and the second author in \cite{Meddour-Zerguine} and enhanced the rate convergence to that \cite{Abidi-Danchin}. 
A successful attempt for anisotropic Boussinesq either for viscosity or diffusivity was recently achieved in \cite{Paicu-Zhu} by Paicu and Zhu. For another connected subject, refer to \cite{Danchin-Zhang,Fanelli,Meddour}.

\subsection{Main results} In this subsection we state the main results of this paper and discuss the headlines of their proofs. Theorem \ref{Th-1} below establishes the global well-posedness topic for the system \eqref{Eq-3} whenever the initial data having a smooth patch structure and the density is a scalar function belongs to certain Besov space space. The second theorem cares with the inviscid limit of the viscous system \eqref{Eq-3} to inviscid one \eqref{Eq-4} when a viscosity goes to zero. In particular, we shall evaluate the rate of convergence between velocities, densities, vortices, and the associate flows.
     
\hspace{0.5cm}The first main result reads as follows. 
\begin{Theo}\label{Th-1} Let $\Omega^0$ be a simply connected bounded domain whose boundary $\partial \Omega^0$ is a Jordan curve in $C^{1+\EE}$ with $0< \EE<1$, and $v_{\mu}^0$ be an initial velocity vector fields in free-divergence which its initial vorticity $\omega_\mu^0={ 1}_{\Omega^0}$.  Let $  \theta_\mu^0 \in {L^{2}}\cap B_{p,r}^{2-\frac{2}{r}}$ with $(p,r)\in ]2, \infty[\times]1,\infty[$ be such that $\frac1p+\frac2r<2 $, then the following assertions are hold. 
\begin{enumerate}
\item[{\bf(1)}] There exists a small positive constant $\epsilon_0$ such that if
\begin{equation}\label{Assup-Th-1}
\| \kappa(\cdot )-1\|_{L^\infty(\RR)}\leq \epsilon_0,
\end{equation} 
system \eqref{Eq-3} admits a global solution fulffils for any $T>0$ and some $\eta>1$,
\begin{equation*}
(v_\mu,\theta_\mu)\in L^\infty\big([0,T];\Lip\big)\times L^\infty\big([0,T];L^2\big)\cap L^{\eta}\big([0,T];W^{2,p}\big),\quad \| \nabla v_\mu\|_{L^{\infty}}\leq  C_0e^{C_0t^{8}}. 
\end{equation*}
Besides, if $p$ and $r$ satisfy $\frac1p+\frac1r\le 1$, then the solution is unique.
\item[{\bf(2)}] The boundary of the transported  domain $\Omega(t)\triangleq \Psi_\mu(t,\Omega^0)$ is $C^{1+\epsilon} $ for all $ t\geq0$ with  $\Psi_\mu$ denotes the flow associated to velocity $v_\mu$.
\end{enumerate}
\end{Theo}
\hspace{0.5cm}Some comments are listed in the following remarks.
\begin{remas}\label{rema theo1 }
\begin{itemize}
\item[--] When the viscosity $\mu =0$, we find the same result as  in  \cite{Paicu-Zhu} for the system \eqref{Eq-4}.

\item[--] The growth of the gradient velocity for the system \eqref{Eq-3} is strongly increases compared to the classical Boussinesq system ($\kappa(\theta)=\kappa$) recently studied in \cite{Meddour-Zerguine} due to the $L^\infty-$estimate of the vorticity,
\begin{equation*}
\Vert \omega_\mu(t)\Vert_{L^\infty} \le C_0(1+t)^{7},
\end{equation*} 
which not being optimal. To my knowledge, the optimality comes from evolution $\theta_{\mu}$ in time, which remains a fruitful field for exploration as we have developed in \cite{Hmidi-Zerguine}.
\item[--] The previous theorem is restrictive because we don't clear the connection between the Lipschitz norm of the velocity and the striated regularity of the initial vorticity $\partial_{X_0}\omega_0$ in $C^{\EE-1}$. We will state the general version in section \ref{S-V-P}.

\item[--] According to \cite{Paicu-Zhu}, the system \eqref{Eq-3} also admits a unique global solution \`a la Yudovich because the presence of $-\mu\Delta$ in the system \eqref{Eq-3} contributes a more regularity.
\end{itemize}
\end{remas}

The hinge phase in the proof of Theorem \ref{Th-1} is to bound the Lipschitz norm of the velocity $\nabla v$ in $L^1_tL^\infty$ with respect to the striated or co-normal regularity of the vorticity $\omega$ in anisotropic H\"older space $C^{\EE}(X)$ spaces by means of logarithmic estimate. The benefits of the family $X=(X_t)$ would involve further factors. Among them that evolves the inhomogeneous transport equation $$\partial_t X_t+v\cdot\nabla X_t=\partial_{X_t}v$$ and commutes with the transport operator $\partial_t+v\cdot\nabla$ in the sense that $[X,\partial_t+v\cdot\nabla]=0$, with $[\cdot,\cdot]$ refers to Lie bracket, see, Section \ref{S-V-P} below.

\hspace{0.5cm}In our situation, the matters may be quite different and contribute to technical difficulties due to the presence term $-\mu\Delta$ in the $v-$equation. Indeed, applying the directional derivative $\partial_{X_t}$ to $\omega_\mu$ in the system \eqref{Eq-5} to obtain
\begin{equation*}
\big(\partial_t+v\cdot\nabla-\mu\Delta\big)\partial_{X_t}\omega_\mu=-\mu[\Delta, \partial_{X_t}]\omega_{\mu}+\partial_{X_t}\partial_1\theta_\mu.
\end{equation*}
To surmount these difficulties, we treat the additional term $\mu[\Delta, \partial_{X_t}]\omega_{\mu}$ as in \cite{Danchin0,Hmidi-1} for Navier-Stokes equations in two dimension of spaces. Eventhough, the term $\partial_{X_t}\partial_1 \theta_\mu$ can be down by applying an elementary estimate of the commutator $\partial_{X_t}\partial_1\theta_\mu=\partial_1\partial_{X_t}\theta_\mu+[\partial_{{X_t}},\partial_1]\theta_\mu$.



\hspace{0.5cm}The second main result discusses the inviscid limit between velocities, densities and vortices and estimate the rate of convergence. Especially, we will prove the following theorem. 
\begin{Theo}\label{Theo1.2}
Let $(v_{\mu}, \rho_{\mu})$, $(v, \rho)$, $(\omega_{\mu}, \rho_{\mu})$ and  $(\omega,\rho)$ be the solution of the \eqref{Eq-3}, \eqref{Eq-4},\eqref{Eq-5}, and \eqref{Eq-50} respectively with the same initial data  satisfies the condition of Theorem \ref{Th-1} such that  $\omega_\mu^0=\omega^0= \mathbf{1}_{\Omega_0}$ where $ \Omega_0 $ is simply connected bounded domain. Then for all $t\ge0, \mu\in]0,1[$ and $p\in[2,+\infty[$ the following assertions hold true.
\begin{itemize}
\item[{\bf(1)}] $\Vert v_\mu(t) -v (t)\Vert_{L^p} + \Vert \theta_\mu (t)-\theta(t) \Vert_{L^p } \leq C_0e^{\exp{C_0t^{8} }} (\mu t)^{\frac{1}{2}+\frac{1}{2p}} $, 
\item[{\bf(2)}] $\Vert \omega _\mu(t) -\omega (t)\Vert_{L^p} \leq C_0e^{\exp{C_0t^{8} }} (\mu t)^{\frac{1}{2p}}$, 
\item[{\bf(3)}] If $\Psi_\mu $ and $\Psi $ denote the flow associated to $v_\mu$ 
and $v$ respectively  then we have
$$ \Vert \Psi _\mu (t) -\Psi (t)\Vert_{L^\infty} \leq C_0e^{\exp{t^{8} }} (\mu t)^{\frac{1}{4}} .$$
\end{itemize}
\end{Theo}
\begin{Rema} The value of the rate convergence already obtained in the previous theorem is the same as in the classical Boussinesq one $\kappa(\theta)=\kappa$ see, \cite{Meddour-Zerguine} because the conditions \eqref{Eq-2} and \eqref{Assup-Th-1} are imposed to prohibit the violent nonlinearity of $\kappa(\cdot)$.   
\end{Rema}
The proof of Theorem \ref{Theo1.2} will be done by exploring \cite{Meddour-Zerguine}, namely some classical ingredients like $L^p-$estimates, the continuity of Riesz transform, complex interpolation results, and the maximal smoothing effects for the density and the vorticity. Moreover, we will also exploit the two conditions \eqref{Eq-2} and \eqref{Assup-Th-1}.\\

{\bf Organization of the paper}. In section 2, we gather the essential background freely used throughout this paper. We embark on some functions spaces and an outline about Littlewood-Paley theory, in particular, the decomposition of unity, the cut-off operators, and paradiffrential calculus following Bony and stated the definition of Besov spaces. Next, we focus on practical results like the persistence regularity for the transport-diffusion equation, the maximal smoothing effect, as well as some properties of the heat kernel. Section 3, concerns the setting of smooth vortex patch, where we start by the push-forward of a vector field in free-divergence and some related properties like the commutation with the transport operator and construct an adequate geometry to be able understanding the vortex patch topic and furnishes the stationary logarithmic estimate which connects the Lipschitz norm of the velocity and the striated regularity of its vorticity. We end this section with a package of a priori estimates for differents quantities in several functional spaces and discuss in detail the proof of Theorem \ref{Th-1} in a more general case. In section 4, we treat the inviscid limit between the two systems \eqref{Eq-3} and \eqref{Eq-4} once the viscosity parameter goes to zero and evaluate the rate of convergence.
\section{Setup and technical toool box}
All throught this work, we designate by $C$ a positive constant which may be different in each occurrence but it does not depend on the initial data. We shall sometimes alternatively use the notation $X\lesssim Y$ for an inequality of the type $X\le CY$ with $C$ is independent of $X$ and $Y$. The notation $C_0$ means a constant depending on the involved norms of the initial data.

\subsection{Function spaces} We embark this section by some definition of H\"older spaces $C^{n+\alpha}$ and Sobolev spaces of type $W^{1,p}$ which will be useful in our analysis. For $\alpha\in]0,1[$ define $C^{\alpha}$ as the set of $u\in L^\infty$ such that
\begin{equation}\label{sp:1}
\|u\|_{C^{\alpha}}=\|u\|_{L^\infty}+\sup_{x\ne y}\frac{|u(x)-u(y)|}{|x-y|^{\alpha}}<\infty.
\end{equation} 
For the limit case $\alpha=1$, the corresponding set is the Lipschitz class which denoted by $\Lip$,
\begin{equation*}
\|u\|_{\Lip}=\|u\|_{L^\infty}+\sup_{x\ne y}\frac{|u(x)-u(y)|}{|x-y|}<\infty.
\end{equation*} 
We will also make use of the space $C^{1+\alpha}$ which is the set of continuously differentiable functions $u$ such that
\begin{equation*}
\|u\|_{C^{1+\alpha}}=\|u\|_{L^\infty}+\|\nabla u\|_{C^\alpha}<\infty.
\end{equation*}
By the same way we can define generally the spaces $C^{n+\alpha}$, with $n\in\NN$ and $\alpha\in]0,1[$.\\
The Sobolev class $W^{1,p}$ for $p\in[1,\infty]$ is the set of tempered distribution $u\in\mathcal{S}'$ endowed with the norm
\begin{equation*}
\|u\|_{W^{1,p}}=\|u\|_{L^p}+\|\nabla u\|_{L^p}.
\end{equation*}
\hspace{0.5cm}Next, we outline some elements about Littlewood-Paley theory will be required in several steps. Let$(\chi,\varphi)\in\mathscr{D}(\RR^2)\times \mathscr{D}(\RR^2)$
be a radial cut-off functions be such that $\supp\chi\subset\{\xi\in\RR^2: \|\xi\|\le1\}$ and $\supp\varphi(\xi)\subset\{\xi\in\RR^2: 1/2\le\|\xi\|\leq 2\}$, so that
\begin{equation*}
\chi(\xi)+\sum_{q\ge0}\varphi(2^{-q}\xi)=1.
\end{equation*}
Through $\chi$ and $\varphi$, the Littlewood-Paley or frequency cut-off operators $(\Delta_q)_{q\ge-1}$ and $(\dot\Delta_q)_{q\ge-1}$ are defined for $u\in \mathscr{S}'(\RR^2)$   
\begin{equation*}
\Delta_{-1} u=\chi(\DD)u,\; \Delta_qu=\varphi(2^{-q}\DD)u\;\;\mbox{for}\;\;q\in\NN,\quad \dot\Delta_qu=\varphi(2^{-q}\DD)u\;\; \mbox{for}\;\; q\in\ZZ. 
\end{equation*}
where in general case $f(\DD)$ stands the pseudo-differential operator $u\mapsto\mathscr{F}^{-1}(f \mathscr{F}u)$ with constant symbol.  The lower frequencies sequence $(S_{q})_{q\ge0}$ is defined for $q\ge0$, 
\begin{equation*}
S_{q}u\triangleq\sum_{j\le q-1}\Delta_{j}u.
\end{equation*}     
In accordance of the previous properties we derive the well-known decomposition of unity
$$
u=\sum_{q\ge-1}\Delta_q u,\quad u=\sum_{q\in\ZZ}\dot\Delta_q u.
$$
The results  currently available allow us to define the inohomogeneous Besov denoted $B_{p,r}^s $ (resp. $\dot B_{p,r}^s$) and defined in the following way.
\begin{defi} For $(p,r,s)\in[1,  +\infty]^2\times\RR $, the inhomogeneous Besov spaces $B_{p,  r}^s$ (resp. homogeneous Besov spaces $\dot B_{p,  r}^s$) are defined by 
$$
B_{p,  r}^s=\{u\in \mathcal{S}'(\RR^2):\|u\|_{B_{p,  r}^s} <+\infty\}, \quad\dot B_{p,  r}^s=\{u\in  \mathcal{S}'(\RR^2)_{|\mathbb{P}}:\|u\|_{\dot B_{p,  r}^s} <+\infty\},
$$
where $\mathbb{P}$ refers to the set of polynomial functions in $\RR^2$ so that
$$
\|u\|_{B_{p,r}^s}\triangleq\left\{\begin{array}{ll}
\Big(\sum_{q\ge-1}2^{rqs}\|\Delta_q u\|_{L^{p}}^r\Big)^{1/r} & \textrm{if $r\in[1, +\infty[,$}\\
\sup_{q\ge-1}2^{qs}\Vert\Delta_q u\Vert_{L^p} & \textrm{if $r=+\infty$.}
\end{array}
\right.
$$
and 
$$
 \|u\|_{\dot B_{p,r}^s}\triangleq\left\{\begin{array}{ll}
\Big(\sum_{q\in\ZZ}2^{rqs}\|\dot\Delta_q u\|_{L^{p}}^r\Big)^{1/r} & \textrm{if $r\in[1, +\infty[,$}\\
\sup_{q\in\ZZ}2^{qs}\Vert\dot\Delta_q u\Vert_{L^p} & \textrm{if $r=+\infty$.}
\end{array}
\right.
$$
\end{defi}
In particular, the spaces $B_{2,2}^s$ coincide with the classical Sobolev spaces $H^s$, whereas for $s\in\RR_{+}\backslash\NN$ the spaces $B^s_{\infty,\infty}$ close to the well-known H\"older spaces $C^s$ defined in particular case by \eqref{sp:1}.

\hspace{0.5cm}The celebrate Bernstein's inequalities are listed in the following lemma. 
\begin{lem}\label{Bernstein} There exists a constant $C>0$ such that for $1\le a\le b\le\infty$, for every function $u$ and every $q\in\NN\cup\{-1\}$, we have
\begin{enumerate}
\item[{\bf(i)}] $\sup_{\vert\alpha\vert=k}\Vert\partial^{\alpha}S_{q}u\Vert_{L^{b}}\le C^{k}2^{q(k+2({1}/{a}-{1}/{b}))}\Vert S_{q}u\Vert_{L^{a}}$.
\item[{\bf(ii)}] $C^{-k}2^{qk}\Vert\Delta_{q}u\Vert_{L^{a}}\le\sup_{\vert\alpha\vert=k}\Vert\partial^{\alpha}\Delta_{q}u\Vert_{L^{a}}\le C^{k}2^{qk}\Vert\Delta_{q}u\Vert_{L^{a}}$.
\end{enumerate}
\end{lem}
A nothworthy consequence of Bernstein's lemma  concerns the embedding relations given by the follwoing.
\begin{prop}\label{Emb} For $(s,\widetilde{s},p,p_1,p_2,r_1,r_2)\in\RR^2\times]1,\infty[\times[1,\infty]^4$ with $\widetilde{s} \le s, p_1\le p_2$ and $r_1\le r_2$, then we have
\begin{enumerate}
\item[{\bf(i)}] $B_{p,r}^{s} \hookrightarrow B_{p,r}^{\widetilde{s}} $.  
\item[{\bf(ii)}] $B_{p_{1},r_1}^{s}(\RR^2)  \hookrightarrow B_{p_2,r_2}^{s+2(1/p_2 -1/p_1)}(\RR^2)$.
\item[{\bf(iii)}] $B_{p,\min(p,2)}^{s} \hookrightarrow W^{s,p}   \hookrightarrow B_{p,\max(p,2)}^{s}$.  
\end{enumerate}
\end{prop}
\hspace{0.5cm}Now, we state Bony's decomposition \cite{Bony} which allows us to split formally the product of two tempered distributions $u$ and $v$ into three pieces. More precisely, we have.
\begin{defi}\label{def:1} For a given $u, v\in\mathcal{S}'$ we have
 $$
uv=T_u v+T_v u+\mathscr{R}(u,v),
$$
with
$$T_u v=\sum_{q}S_{q-1}u\Delta_q v,\quad  \mathscr{R}(u,v)=\sum_{q}\Delta_qu\widetilde\Delta_{q}v  \quad\hbox{and}\quad \widetilde\Delta_{q}=\Delta_{q-1}+\Delta_{q}+\Delta_{q+1}.
$$
\end{defi}
The mixed space-time spaces are stated as follows. 
\begin{defi} Let $T>0$ and $(s,\beta,p,r)\in\RR\times[1, \infty]^3$.  We define the spaces $L^{\beta}_{T}B_{p,r}^s$ and $\widetilde L^{\beta}_{T}B_{p,r}^s$ respectively by: 
$$
L^\beta_{T}B_{p,r}^s\triangleq\Big\{u: [0,T]\to\mathcal{S}^{'}; \Vert u\Vert_{L_{T}^{\beta}B_{p, r}^{s}}=\big\Vert\big(2^{qs}\Vert \Delta_{q}u\Vert_{L^{p}}\big)_{\ell^{r}}\big\Vert_{L_{T}^{\beta}}<\infty\Big\},
$$
$$
\widetilde L^{\beta}_{T}B_{p,r}^s\triangleq\Big\{u:[0,T]\to\mathcal{S}^{'}; \Vert u\Vert_{\widetilde L_{T}^{\beta}{B}_{p, r}^{s}}=\big(2^{qs}\Vert \Delta_{q}u\Vert_{L_{T}^{\beta}L^{p}}\big)_{\ell^{r}}<\infty\Big\}.
$$
The relationship between these spaces is given by the following embeddings. Let $ \varepsilon>0,$ then 
\begin{equation}\label{embeddings}
\left\{\begin{array}{ll}
L^\beta_{T}B_{p,r}^s\hookrightarrow\widetilde L^\beta_{T}B_{p,r}^s\hookrightarrow L^\beta_{T}B_{p,r}^{s-\varepsilon} & \textrm{if  $r\geq \beta$},\\
L^\beta_{T}B_{p,r}^{s+\varepsilon}\hookrightarrow\widetilde L^\beta_{T}B_{p,r}^s\hookrightarrow L^\beta_{T}B_{p,r}^s & \textrm{if $\beta\geq r$}.
\end{array}
\right.
\end{equation}
\end{defi}

The following result is a consequence of paradifferential calculus which will be beneficial later.  
\begin{cor}\label{Coro:1} Given $\EE\in]0,1[$ and $X$ be a vector field be such that $X, \Div X\in C^\EE$. Then for $f$ be a Lipschitz scalar function $k\in\{1,2\}$ the following statement holds.
\begin{equation*}
\|(\partial_k X)\cdot\nabla f\|_{C^{\EE-1}}\le C\|\nabla f\|_{L^\infty}\big(\|\Div X\|_{C^\EE}+\|X\|_{C^\EE}\big).
\end{equation*}
\end{cor}
\begin{proof} Exploring Bony's decomposition stated in Definition \eqref{def:1} to write 
\begin{equation}\label{eq:Ch-Z}
(\partial_k X)\cdot\nabla f=\sum_{i=1}^{2} T_{\partial_k X^i}\partial_i f+T_{\partial_i f}\partial_k X^i+\mathscr{R}(\partial_k X^i,\partial_i f).
\end{equation}
Moreover by  definition  we have 
\begin{equation*}
\mathscr{R}(\partial_k X^i,\partial_i f)=\sum_{q\geq-1} \Delta_q\partial_k X^i \widetilde{\Delta}_q \partial_i f.
\end{equation*}
Using the fact $\Delta_q\partial_k X^i \widetilde{\Delta}_q \partial_i f =\partial_k(\Delta_q X^i \widetilde{\Delta}_q \partial_i f)-\Delta_q X^i \widetilde{\Delta}_q \partial_k \partial_i f$ we get 
\begin{equation*}
\mathscr{R}(\partial_k X^i,\partial_i f)= \partial_k \big( \sum_{q\geq-1}(\Delta_q X^i \widetilde{\Delta}_q \partial_i f \big)-\sum_{q\geq-1}\Delta_q X^i \widetilde{\Delta}_q \partial_k \partial_i f .
\end{equation*} 
It follows that
\begin{equation*}
\mathscr{R}(\partial_k X^i,\partial_i f)= \partial_k \big( \sum_{q\geq-1}(\Delta_q X^i \widetilde{\Delta}_q \partial_i f \big)- \partial_i \big( \sum_{q\geq-1}\Delta_q X^i \widetilde{\Delta}_q \partial_k  f\big) + \sum_{q\geq-1} \Delta_q \partial_i X^i  \widetilde{\Delta}_q \partial_k f .
\end{equation*} 
Plug the last estimate in \eqref{eq:Ch-Z} we find 
\begin{equation*}\label{eq:Ch-Z}
(\partial_k X)\cdot\nabla f=\sum_{i=1}^{2} \Big( T_{\partial_k X^i}\partial_i f+T_{\partial_i f}\partial_k X^i+\partial_k \mathscr{R}( X^i,\partial_i f)- \partial_i \mathscr{R}(X^i,\partial_k  f) \Big) + \mathscr{R}(\Div X,\partial_k) .
\end{equation*}
Taking the $C^{\EE -1}-$ norm  to above equation after applying Lemma 2  page 6 in \cite{HH}  with $s=\EE-1 $ infer that 
\begin{eqnarray*}
\|(\partial_k X)\cdot\nabla f\|_{C^{\EE -1}} \lesssim  \|\nabla f\|_{L^\infty} \| X\|_{C^{\EE }} + \|\nabla f\|_{L^\infty} \|\Div X\|_{C^{\EE }}.
\end{eqnarray*}
This completes the proof of the Corollary \label{Coro:1}.

\end{proof}


Accordingly, we have the following interpolation result.
\begin{cor}
Let $T>0,\; s_1<s<s_2$ and $\zeta\in(0, 1)$ such that $s=\zeta s_1+(1-\zeta)s_2$. Then we have
\begin{equation}\label{m1}
\Vert u\Vert_{\widetilde L_{T}^{a}{B}_{p, r}^{s}}\le C\Vert u\Vert_{\widetilde L_{T}^{a}{B}_{p, \infty}^{s_1}}^{\zeta}\Vert u\Vert_{\widetilde L_{T}^{a}{B}_{p, \infty}^{s_2}}^{1-\zeta}.
\end{equation}
\end{cor}
\subsection{Practical results} This subsection motivates by some preparatory results freely used throughout our analysis. Let us denote that the most results relative to the system depend mainly on a priori estimates in Besov space for the following transport-diffusion equation.
\begin{equation}\label{PC1.0}
\left\{ \begin{array}{ll}
\partial_{t}a + v\cdot\nabla a -\mu\Delta a=g,&\\ 
a_{\mid t=0}=a^{0}. 
\end{array} 
\right.
\end{equation}  
We start with the persistence of Besov regularity for \eqref{PC1.0} whose proof may be found for example in  \cite{Bahouri-Chemin-Danchin}
\begin{prop}\label{prop2.9}
Let $(s,r,p) \in ]-1, 1[ \times[1, \infty]^{2}$ and $v$ be a smooth vector field in free-divergence. Assume that $(a^{0},g)\in B_{p,r}^{s}\times L_{\loc}^{1}(\mathbb{R}_+; B_{p,r}^{s})$. Then for every smooth solution $a$ of $\eqref{PC1.0}$ and $t\geq 0$ we have
\begin{align*}
\Vert a(t) \Vert_{B_{p,r}^{s}}\leq Ce^{CV(t)}\left( \Vert a^{0}\Vert_{B_{p,r}^{s}}+ \int_0^{t}e^{-CV(\tau)}\Vert g(\tau)\Vert_{B_{p,r}^{s}}d\tau\right), 
\end{align*}
with the notation 
\begin{align*}
V(t)=\int_0^{t}\Vert\nabla v(\tau) \Vert_{L^{\infty}}d\tau,
\end{align*}
where $C=C(s)$ being a positive constant.
\end{prop}

The statement of maximal regularity for \eqref{PC1.0} in mixed space-time Besov space in given by the following result. For its proof, see \cite{Bahouri-Chemin-Danchin,Hmidi-Keraani2}.
\begin{prop}\label{prop2.10}
Let $(s,p_1,p_2,r)\in ]-1, 1[\times[1,\infty]^{3}$ and $v$ be a free-divergence vector field belongs to $L_{loc}^{1}(\mathbb{R}_+;\Lip)$ then there exists a constant $C\geq0$, so that for every smooth solution $a$ of \eqref{PC1.0} we have for all $ t\geq 0 $
\begin{equation*}
\mu^{\frac{1}{r}}\Vert a\Vert_{\widetilde{L}_t^{r}B_{p_{1},p_{2}}^{{s+\frac{2}{r}}}}\leq Ce^{CV(t)}(1 + \mu t)^{\frac{1}{r}}\Big(\Vert a^{0}\Vert_{B_{p_{1},p_{2}}^{{s}}} + \Vert g\Vert_{L_t^{1}B_{p_1,p_2}^{{s}}}\Big).
\end{equation*}
\end{prop}
 
Due to $\Div v=0$, we derive via stream function the Biot-Savart $v=\nabla^{\perp}\Delta^{-1}\omega$, combined with Cald\'eron-Zygmund analysis, we may deduce
\begin{prop}\label{C.Z}
Let $p\in ]1, +\infty[$ and $v$ be a free-divergence vector field whose vorticity $\omega \in L^{p}$. Then $\nabla v \in L^{p}$ and  
\begin{align*}
\Vert \nabla v\Vert_{L^{p}}\leq C \frac{p^{2}}{p-1}\Vert \omega\Vert_{L^{p}}.
\end{align*}
with $ C $ being a universal constant.
\end{prop}



A worthwile property of regularization of the heat semigroup in mixed time-space spaces is given by the following result. The proof can be found in Theorem 7.3, \cite{Lemarie-Rieusset}.

\begin{prop}\label{op A}
Let $(r,p) \times ]1,+\infty[^2 $ and  $\mathcal{A}$ be an operator defined by
\begin{equation*}
\mathcal{A}a= \int_0^t \nabla^2 \mathbb{S}(t-\tau)a(\tau,\cdot) d\tau.
\end{equation*}   
Then $\mathcal{A}$ is a bounded from $L^r_tL^p$ into itself, that is for $a\in L^r_tL^p$ we have
\begin{equation*}
\| \mathcal{A}a \Vert_{L^r_t L^{p}}\leq C \Vert a\Vert_{L^r_tL^{p}}.
\end{equation*}
\end{prop}

The connextion between Besov norms with negative indices and the heat semigroup is stated in the following lemma, see, e. g. Theorem 2. 34 in \cite{Bahouri-Chemin-Danchin}.
\begin{prop}\label{A 1}
Let $(s,p',r) \in ]0,+\infty [ \times [1,+\infty]^2$. A constant $C$ exists which satisfies
\begin{equation*}
C^{-1} \Vert a \Vert_{\dot{B}_{{p',r}}^{-2s}} \le\big\| \|t^{s} \mathbb{S}(t) a\|_{L^p}\big \|_{L^{r}\big(\RR_+;\frac{dt}{t}\big)} \leq C \Vert a \Vert_{\dot{B}_{{p',r}}^{-2s}}. 
\end{equation*}
\end{prop}

\section{Smooth vortex patch }\label{S-V-P}
This section addresses actually to treat the smooth vortex in the setting of thermal time-dependent diffusivity. We present the necessary background on the admmissible family of vector field rises explicitly in our study.    
\subsection{Push-forward: definitions and properties} Let $X:\RR^N\rightarrow\RR^N$ be a smooth family of a vector fields and $f:\RR^N\rightarrow\RR$ be a smooth function. The derivative of $f$ in the direction $X$ is denoted by $\partial_X f$ and defined by
\begin{equation*}
X(f)=\partial_X f=\sum_{i=1}^{N}X^{i}\partial_{i}f=X\cdot\nabla f.
\end{equation*}
This is the Lie derivative of the function $f$ with respect to the vector field $X$, denoted usually by $\mathcal{L}_X f$ and in the previous formula we adopt different notations for this object.

\begin{defi}\label{def3.1} 
Let $X,Y:\RR^{N}\rightarrow\RR^{N}$ be a two family of vector fields. Their commutator is defined as the Lie bracket $[X,Y]$ which is given in the coordinates system by
\begin{eqnarray*}
[X,Y]^{i}&=&\sum_{j=1}^{N}(X^{j}\partial_{j}Y^{i}-Y^{j}\partial_{j}X^{i})\\
&=&\partial_X Y^{i}-\partial_Y X^{i}.
\end{eqnarray*} 
\end{defi}
We observe that the previous identity can also be written in the following form
\begin{equation*}
\partial_X\partial_Y-\partial_Y\partial_X=\partial_{\partial_X Y-\partial_Y X}.
\end{equation*}
For $f\in L^\infty$ and $X$ a family of vector fields we define $\partial_X f$ in a weak sense as
\begin{equation}\label{wd}
\partial_X f=\Div(Xf)-f\Div X.
\end{equation}
Next, we state the definition of the well-known {\it push-forward} of a family of vector fields $X$ by a diffeomorphism. To be precise we have.
\begin{defi}\label{def3.2}  Let $X:\RR^N\rightarrow\RR^N$ be a family of vector field and $\phi$ be a diffeomorphism of $\RR^N$. The push-forward of $X$ by $\phi$, denoted by $\phi_{\star}X$ is defined by
\begin{equation*}
(\phi_{\star}X)(x)=\big(X\cdot\nabla \phi\big)(\phi^{-1}(x)).
\end{equation*}
\end{defi} 

In particular case, where $X$ is replaced by a time-dependent vector field $v(t)$ in Lipschitz class. It is classical that this latter induces a diffeomorphism flow $\Psi(t,\cdot)$ given by \eqref{Eq-7}. 
Consequently, the push-forward for a given family of vector fields $X_0$ by the flow $\Psi(t,\cdot)$ is the time-dependent family of vector fields $(X_t)$ that can be written in the local coordinates as follows:
\begin{equation}\label{pf}
X_t(x)=\big(X_0\cdot\nabla\Psi(t,x)\big)\big(\Psi^{-1}(t,x)\big).
\end{equation}
The first important property of such family is that it evolves the following inhomogeneous transport equation
\begin{equation}\label{tr}
\partial_t X_t+v\cdot\nabla X_t=X_t\cdot\nabla v.
\end{equation}

Another main feature of the family $(X_{t})$ given by the equation \eqref{tr} reflects in its commutation with the transport operator $D_t=\partial_t+v\cdot\nabla$. This implies an important consequence about the dynamics of the tangential regularity of the vorticity subject to the system \eqref{Eq-5}. Actually, one obtains easily the following result.
\begin{prop}\label{com} Let $X$ be the push-forward of a smooth family of vector fields $X_0$ defined by \eqref{pf}. Then $X$ commutes with the transport operator $D_t=\partial_t+v\cdot\nabla$ in the sense
\begin{equation*}
[X,D_t]=\partial_X D_t-D_t\partial_X=0.
\end{equation*}
\end{prop}


\subsection{Anisotropic spaces}\label{X_t} This subsection motivates by the definition of the anisotropic H\"older spaces which is considered as a cornerstone to reach the Lipschitz norm of the velocity.
\begin{defi}\label{3.1}
Let $\epsilon \in ]0, 1[$. A family of vector fields $X = (X_\lambda)_{\lambda \in \Lambda}$ is said to be admissible if and only if   the following assertions are hold.\\
\begin{enumerate}
\item[{\bf(i)}]Regularity: $\forall \lambda \in \Lambda \quad X_\lambda, \Div X_\lambda   \in C^{\epsilon}.$
\item[{\bf(ii)}] Non-degeneray: $ I(X)\triangleq \inf_{x\in \mathbb{R}^{N}}\sup_{\lambda \in \Lambda }\mid X_\lambda (x)\mid>0.$
\end{enumerate}
Setting
\begin{equation*}
\widetilde{\|}X_\lambda\|{C^{\epsilon}}\triangleq \| X_\lambda\|_{C^{\epsilon}} + \| \Div X_\lambda\|_{C^{\epsilon}}.
\end{equation*}

\end{defi}
\begin{defi}\label{def:3.2.1}
Let $\epsilon \in ]0, 1[$ and $X $ be an admissible family of vector fields. We say that $u \in C^{\epsilon}(X)$ if and only if:\\
\begin{enumerate}
\item[{\bf(i)}] $ u \in L^{\infty}$ and satisfies
\begin{equation*}
\forall \lambda \in \Lambda, \partial_{X_{\lambda}} u \in C^{\epsilon-1}, \quad \sup_{\lambda \in \Lambda }\Vert\partial_{X_{\lambda}} u\Vert_{C^{\epsilon-1}}< +\infty .
\end{equation*}   
\item[{\bf(ii)}]  $ C^{\epsilon}(X)$ is a normed space with
\begin{equation*}    
\Vert u\Vert_{C^{\epsilon}(X)}\triangleq \frac{1}{I(X)}\left( \Vert u\Vert_{L^{\infty}}\sup_{\lambda \in \Lambda}\widetilde{\Vert}X_\lambda\Vert_{C^{\epsilon}}+ \sup_{\lambda \in \Lambda }\Vert \partial_{X_{\lambda}} u\Vert_{C^{\epsilon-1}}\right) .
\end{equation*} 
\end{enumerate}
To derive the Lipschitzian norm of the velocity it is mandatory to state the following stationary logarithmic estimate which its original proof can be found in J. Chemin \cite{Chemin2}. 
\begin{Theo}\label{The estimate log}
Let $\epsilon \in ]0, 1[$ and $X = (X_{t,\lambda})_{\lambda \in \Lambda}$ be a family of vector fields as in Definition \ref{def:3.2.1}. Let $v$ be a free-divergence vector field such that its vorticity $\omega$ belongs to $ L^{2}\cap C^{\epsilon}(X)$. Then there exists a constant C depending only on $\epsilon $, such that
\begin{align}
\Vert \nabla v\Vert_{L^{\infty}}\leq C\left( \Vert \omega \Vert_{L^{2}} + \Vert \omega\Vert_{L^{\infty}} \log \left( e+ \frac{\Vert \omega \Vert_{C^{\epsilon}(X)}}{\Vert \omega \Vert_{L^{\infty}}} \right) \right) .
\end{align}
\end{Theo}
\end{defi}
We now make a precise interpretation of the boundary regularity and the tangent space which will be explored in the proof of Theorem \ref{Th-1} 
\begin{defi}\label{Char} Let $0<\EE<1$, then we have the following definitions. 
\begin{enumerate}
\item[{\bf(1)}] A closed curve $\Sigma$ is said to be $C^{1+\EE}-$regular if there exists $f\in C^{1+\EE}(\RR^2)$ such that $\Sigma$ is a locally zero sets of $f$, i.e., there exists a neighborhood $V$ of $\Sigma$ such that
\begin{equation}\label{Reg-Om}
\Sigma=f^{-1}\{0\}\cap V,\quad \nabla f(x)\ne 0\quad\forall x\in V.
\end{equation}
\item[{\bf(2)}] A vector field $X$ with $C^\EE-$regularity is said to be tangent to $\Sigma$ if $X\cdot\nabla f_{|\Sigma}=0$. The set of such vector field denoted by $\mathcal{T}_{\Sigma}^{\EE}$.
\end{enumerate}
\end{defi}

Given a compact curve $\Sigma$ of the regularity $C^{1+\EE}$, with $0<\EE<1$. The striated or co-normal space  $C^{\EE}_{\Sigma}$ associated  to $\Sigma$ is defined by
$$
C^{\EE}_{\Sigma}\triangleq\big\{u\in L^\infty(\RR^2); \forall X\in\mathcal{T}^{\EE}_{\Sigma},\;(\Div X=0)\Rightarrow\Div(Xu)\in C^{\EE-1}\big\}.
$$
According Danchin's result \cite{Danchin0}, the class $C^{\EE}_{\Sigma}$ doesn't covers only the vortex patch of the type $\omega_0={\bf 1}_{\Omega_0}$, but also encompass the so-called general vortex. Specifically, we have.
\begin{prop}\label{G-V} Let $\Omega_0$ be a $C^{1+\EE}-$bounded domain, with $0<\EE<1$. Then for every function $f\in C^\EE$, we have
\begin{equation*}
f{\bf 1}_{\Omega_0}\in C^{\EE}_{\Sigma}.
\end{equation*}
\end{prop}


\subsection{A priori estimates} This part is considered as the pivot in our analysis. We shall give some a priori estimates about velocity and vorticity, as well as the temperature in different functional spaces.  
\begin{prop}\label{prop v}
Let $(v_\mu,\theta_\mu)$  be a smooth solution of \eqref{Eq-3} then the following estimates hold true.
\begin{enumerate}
\item[{\bf(i)}] Let $p \in [2,+\infty]$ then for $ t \geq 0$ we have
\begin{equation*}
\Vert \theta_\mu(t) \Vert_{L^p} \leq \Vert \theta_\mu^0 \Vert_{L^p}.
\end{equation*}
\item[{\bf(ii)}] Let $(v_\mu^0,\theta_\mu^0) \in L^2 \times L^2 $ then for $ t \geq 0 $ we have 
$$
\Vert v_\mu(t) \Vert_{L^2} \leq \Vert v_\mu^0 \Vert_{L^2}+t\Vert \theta_\mu^0 \Vert_{L^2}.
$$
and
\begin{equation*}
\Vert v_\mu(t) \Vert_{L^2}^2+ 2\mu \int_0^t||\nabla v_\mu (\tau)||_{L^2}^2 d\tau \leq C_0(1+t^2). 
\end{equation*}
\item[{\bf(iii)}] Let $(\omega_\mu^0, \theta_\mu^0)\in L^2 \times L^2 $ then for $ t \geq 0 $ we have
\begin{equation*} 
\Vert \omega_\mu(t)\Vert_{L^2} \leq  \Vert {\omega_\mu}^0\Vert_{L^2}+ \kappa_0\Vert \theta_\mu^0\Vert_{L^2}+t.
\end{equation*}  
\end{enumerate}
\end{prop}
\begin{proof}
{\bf(i)} Multiplying $\theta_{\mu}-$equation by $|\theta_\mu| ^{p-2} \theta_\mu $, integrating by parts over $\RR^2$ and bearing in mind that $\nabla \cdot v_\mu=0$ leads 
$$
\frac{1}{p} \frac{d}{dt} \Vert \theta_\mu(t) {\Vert}^p_{L^p} + \int_{\mathbb{R}^2} \kappa(\theta_\mu)\nabla \theta_\mu \nabla(|\theta_\mu| ^{p-2} \theta_\mu)dx=0.  
$$
In light to \eqref{Eq-2} we may write
\begin{equation}\label{3.6}
\frac{1}{p} \frac{d}{dt} \Vert \theta_\mu(t) {\Vert}^p_{L^p} + \kappa^{-1}_0(p-1)\int_{\mathbb{R}^2}  |\nabla\theta_\mu| ^2|\theta_\mu|^{p-2}  dx  \leq 0.
\end{equation}
Integrating in time over $[0,t]$ we infer that
\begin{equation*}\label{3.7}
\Vert \theta_\mu(t) {\Vert}_{L^p} \leq \Vert \theta_\mu^0 {\Vert}_{L^p}.
\end{equation*}
The case, where $p=\infty$ is a direct consequence of the maximum principle.\\
{\bf(ii)} The classical scalar product in $L^2$ for $v_{\mu}-$equation allows us to achieve. 
$$
\frac{1}{2} \frac{d}{dt} \Vert v_\mu(t) {\Vert}^2_{L^2} + \mu \int_{\mathbb{R}^2} |\nabla v_\mu|^2 dx  \leq \int_{\mathbb{R}^2} |\theta_\mu v_\mu| dx.
$$
By virtue of Cauchy-Schwarz inequality, we readily get
\begin{equation}\label{ grad v l2}
\frac{1}{2}  \frac{d}{dt}  \Vert v_\mu(t) {\Vert}^2_{L^2} + \mu\|\nabla v_\mu(t)\|^2_{L^2}    \leq   \Vert v_\mu(t) {\Vert}_{L^2}  \Vert \theta_\mu(t) {\Vert}_{L^2}.
\end{equation} 
Thanks to {\bf(i)} for $p=2$, that is $\Vert \theta_\mu(t) {\Vert}_{L^2} \leq \Vert \theta_\mu^0 {\Vert}_{L^2}$ it follows after an integration with respect to time that
$$ 
\Vert v_\mu(t) \Vert_{L^2} \leq \Vert v_\mu^0 \Vert_{L^2}+t\Vert \theta_\mu^0 \Vert_{L^2}.
$$
Finally, integrating in time \eqref{ grad v l2}  and using the last estimate
\begin{eqnarray*}
\Vert v_\mu(t) \Vert_{L^2}^2+ 2\mu \int_0^t||\nabla v_\mu (\tau)||_{L^2}^2 d\tau &\le&  \Vert v^0_\mu \Vert^2_{L^2} +2t\Vert \theta^0_\mu \Vert_{L^2}\Vert v^0_\mu \Vert_{L^2}+t^2\|\theta^0_\mu \Vert^2_{L^2} \\ &\le& 2(\Vert \theta_\mu^0 \Vert_{L^2}+\Vert v^0_\mu \Vert_{L^2})^2(1+t^2). 
\end{eqnarray*}
{\bf(iii)} The classical $L^2-$estimate of $\omega_\mu-$equation
gives

$$
\frac12\frac {d}{dt} \Vert \omega_\mu(t) \Vert_{L^{2}}^2 + \mu \int_{\mathbb{R}^{2}}|\nabla \omega_\mu |^2 dx = \int_{\mathbb{R}^{2}} \partial_1 \theta_\mu \omega_\mu dx.
$$
On account  Cauchy-Schwartz provides   
$$
\frac12\frac{d}{dt}\Vert \omega_\mu(t) \Vert_{L^{2}}^2 +\mu\int_{\mathbb{R}^{2}}|\nabla \omega_\mu |^2 dx\le\|\nabla \theta_\mu(t)\|_{L^2}\|\omega_\mu(t) \|_{L^2}.
$$
Thus we have
$$
\frac12\frac{d}{dt}\Vert \omega_\mu(t) \Vert_{L^{2}} \le\|\nabla \theta_\mu(t)\|_{L^2}.
$$
Integrating in time and employ the Cauchy-Schwarz inequality with respect to time, one has  
$$ \Vert \omega_\mu(t)\Vert_{L^2} \leq  \Vert {\omega_\mu}^0\Vert_{L^2}+t^{\frac{1}{2}}\Vert\nabla\theta_\mu{\Vert}_{L^2_tL^2}  $$
Young's inequality gives
\begin{equation}\label{3.10}
 \Vert \omega_\mu(t)\Vert_{L^2} \leq  \Vert {\omega_\mu}^0\Vert_{L^2}+t+ \Vert\nabla\theta_\mu{\Vert}^2_{L^2_tL^2}. 
\end{equation} 
On the other hand, by exploring \eqref{3.6} for $p=2$, it holds
\begin{equation*}
\frac{1}{2} \frac{d}{dt} \Vert \theta_\mu(t) {\Vert}^2_{L^2} + \kappa^{-1}_0\Vert\nabla\theta_\mu(t){\Vert}^2_{L^2}  \leq 0.
\end{equation*}
It follows from integrating in time that
\begin{equation}\label{grad teta L^2}
\Vert\nabla\theta_\mu{\Vert}^2_{L^2_tL^2}  \leq  \kappa_0\Vert \theta^0_\mu {\Vert}^2_{L^2} .
\end{equation}
Plug the last  estimate in \eqref{3.10} we deduce that
$$\Vert \omega_\mu(t)\Vert_{L^2} \leq  \Vert {\omega_\mu}^0\Vert_{L^2}+ \kappa_0\Vert \theta_\mu^0\Vert^2_{L^2}+t $$ 

This completes the proof of the Proposition \ref{prop v}
\end{proof} 
At this stage, we need to bound the term $\nabla \theta_\mu$ in $L^1_t L^p$ space, that is $\theta_\mu$ in $L^1_t W^{1,p}$. For this reason, we explore the maximal regularity of the heat equation to gain also that $\theta_\mu $ is well-controlled in $ L^\eta W^{2,p}$.
\begin{prop}\label{prop theta}
Let $(v_\mu,\theta_\mu)$  be a smooth solution of  \eqref{Eq-3} which satisfying the assumptions of theorem. Then for $(\eta,p)\in[1,\infty[\times]2,\infty[ $ and $t\geq0$ we have
\begin{equation*}
\theta_\mu\in L^{\eta}_t W^{2,p}.
\end{equation*}
More precisely, 
\begin{equation*}
\| \nabla \theta_\mu\|_{L^{2\eta}_t L^{2p}} \leq C_0(1+t)^{\frac{7}{2}},\quad \|\nabla^2 \theta_\mu\|_{L^\eta_t L^{p}}\leq C _0(1+t)^{7}.
\end{equation*}
\end{prop}
\begin{proof} We start by estimating the quantity $\nabla \theta_\mu$ in $L^{2\eta}_tL^{2p}$. For this purpose we write $\theta-$ variable as a solution of an adequate heat equation in the following way.
\begin{equation*}
\partial_{t}\theta_\mu+v_\mu\cdot\nabla \theta_\mu = \nabla\cdot(\kappa(\theta_\mu)\nabla \theta_\mu).
\end{equation*} 
A straightforward manupilation and on account $\nabla\cdot v_\mu=0$ yield
\begin{equation*}\label{equatioin theta vartion 2}
\partial_{t}\theta_{\mu}-\Delta\theta_{\mu}=-\nabla\cdot(v_\mu \theta_\mu) +\nabla\cdot \big((\kappa(\theta_\mu)-1)\nabla \theta_\mu\big).
\end{equation*}
Since $\Delta$ is a good infinitesimal generator for the heat semigroup $\Ss(t)=e^{t\Delta}$, so $\theta$ is expressed by Duhamel's formula,
\begin{equation}\label{Duh-form}
\theta_\mu (t,x)=  \Ss(t)\theta^0_\mu(x) -\int_0^t \mathbb{S}(t-\tau) \nabla\cdot(v_\mu  \theta_\mu)(\tau,x) d \tau+\int_0^t  \mathbb{S}(t-\tau) \nabla\cdot((\kappa(\theta_\mu)-1)\nabla \theta_\mu) (\tau,x) d \tau.
\end{equation}
Apply $\nabla$ to this formula to obtain
\begin{eqnarray}\label{F thita }
\nabla\theta_\mu (t,x)&=& \nabla\mathbb{S}(t) \theta^0_\mu(x) -\int_0^t \nabla \mathbb{S}(t-\tau)\big(\nabla\cdot(v_\mu   \theta_\mu)(\tau,x)\big) d \tau\\
\nonumber&&+\int_0^t  \nabla  \mathbb{S}(t-\tau)\nabla\cdot\big((\kappa(\theta_\mu)-1)\nabla \theta_\mu) (\tau,x)\big) d \tau. 
\end{eqnarray} 
Taking the $L^{2\eta}_tL^{2p}-$norm to \eqref{F thita }, so by integrating by parts over $\RR^2$, it follows
\begin{eqnarray} \label{I-0+I-1+I-2}
\Vert \nabla\theta_\mu  \Vert_{L^{2\eta}_tL^{2p}} &\leq& \big\Vert \nabla\mathbb{S}(t) \theta^0_\mu  \big\Vert_{L^{2\eta}_tL^{2p}}+\bigg\| \int_0^t  \nabla^2 \mathbb{S}(t-\tau) (v_  \mu   \theta_\mu) d\tau \bigg\|_{L^{2\eta}_t L^{2p}}
\nonumber \\&+&\bigg\| \int_0^t   \nabla^2 \mathbb{S}(t-\tau) ((\kappa(\theta_\mu)-1)\nabla \theta_\mu) d\tau \bigg\|_{L^{2\eta}_tL^{2p}} 
\nonumber \\
&\triangleq& \mathrm{I}_0+\mathrm{I}_1+\mathrm{I}_2.
\end{eqnarray}  
To treat $\mathrm{I}_0$ we require to apply Proposition \ref{A 1} with $s=\frac{1} {2\eta},r=2\eta$ and $p'=2p$ to write
\begin{equation*}
\mathrm{I}_0\triangleq\big\Vert \nabla\mathbb{S}(t) \theta^0_\mu  \big\Vert_{L^{2\eta}_tL^{2p}} \lesssim  \Vert\nabla\theta^0_\mu  \Vert_{\dot{B}_{{2p,2\eta}}^{-\frac{1}{\eta}}}.  
\end{equation*}
The fact that $\eta\ge1$ implies $2\eta\geq \eta$, so we have $ {\dot{B}_{{2p,\eta}}^{-\frac{1}{\eta}}} \hookrightarrow {\dot{B}_{{2p,2\eta}}^{-\frac{1}{\eta}}}$. Thus, in view of the continuity $\nabla: {\dot{B}_{{2p,\eta}}^{1-\frac{1}{\eta}}}\rightarrow {\dot{B}_{{2p,\eta}}^{-\frac{1}{\eta}}}$, we infer that
\begin{equation}\label{I-0} 
\mathrm{I}_0 \lesssim   \Vert\nabla\theta^0_\mu  \Vert_{\dot{B}_{{2p,\eta}}^{-\frac{1}{\eta}}} \lesssim \Vert \theta^0_\mu  \Vert_{\dot{B}_{{2p,\eta}}^{1-\frac{1}{\eta}}}
\lesssim \Vert \theta^0_\mu  \Vert_{B_{{2p,\eta}}^{1-\frac{1}{\eta}}}.
\end{equation}
In the setting, where $ \frac{1}{p} +\frac{1}{r} \leq 1$, we have immediately $r\geq \frac{p}{p-1}>1$. By taking $\eta =r $, then in accordance with  {\bf{(i)}}  and {\bf{(ii)}} in Proposition \ref{Emb}, one gets 
\begin{equation*}
B^{2-\frac{2}{r}}_{p,r}\hookrightarrow B^{1-\frac1r+\frac{1}{p}}_{p,r}\hookrightarrow B^{1-\frac{1}{r}}_{2p,r}.
\end{equation*}

For the case $ \frac{1}{p} +\frac{1}{r}>1$ we have $r> \eta>1$, with $\frac{1}{\eta}=\frac{1}{p} +\frac{2}{r}-1$, again {\bf{(i)}} and {\bf{(ii)}} in Proposition \ref{Emb} provide the embeddings 
\begin{equation*}
{B_{{p,r}}^{2-\frac{2}{r}}} \hookrightarrow {B_{{2p,r}}^{2-\frac{1}{p}-\frac{2}{r}}} \hookrightarrow {B_{{2p,r}}^{1-\frac{1}{\eta}}}.
\end{equation*}
Combining the last two assertions and plug them in \eqref{I-0}, it holds 
\begin{equation}\label{I-0-f}
\mathrm{I}_0 \lesssim  \Vert \theta^0_\mu  \Vert_{B_{{p,r}}^{2-\frac{2}{r}}}.
\end{equation}
For the term $\mathrm{I}_1$, we explore Proposition \ref{op A} and H\"older's inequality to conclude
\begin{eqnarray}\label{I-1}
\mathrm{I}_1\triangleq \bigg\| \int_0^t  \nabla^2 \mathbb{S}(t-\tau) (v_  \mu  \cdot \theta_\mu) d\tau \bigg\|_{L^{2\eta}_t L^{2p}} &\lesssim& \Vert v_\mu \theta_\mu  \Vert_{L^{2\eta}_tL^{2p}}\\
&\lesssim&\|v_{\mu}\|_{L^\infty_t L^\infty}\|\theta_{\mu}\|_{L_t^{2\eta}L^{2p}}.\nonumber 
\end{eqnarray}
Let us denote that the term $\Vert v_\mu \Vert_{L^{\infty}_tL^{\infty}}$ can be done by employing Gagliardo-Nirenberg inequality and 
Cald\'eron-Zygmung inequatlity stated in a Proposition \ref{C.Z}, i.e. for $p\in]2,\infty[$ we have
\begin{eqnarray*} 
\Vert v_\mu \Vert_{L^{\infty}_tL^{\infty}}  & \lesssim & \Vert v_\mu \Vert_{L^{\infty}_tL^{2}} ^{\frac{p-1}{2p-1}} \Vert \nabla v_\mu  \Vert_{L^{\infty}_tL^{2p}}^{\frac{p}{2p-1}} \\
& \lesssim &  \Vert v_\mu \Vert_{L^{\infty}_tL^{2}} ^{\frac{p-1}{2p-1}}   \Vert \omega_\mu  \Vert_{L^{\infty}_tL^{2p}}^{\frac{p}{2p-1}}. 
\end{eqnarray*}
For the term $\Vert \omega_\mu  \Vert_{L^{\infty}_tL^{2p}}^{\frac{p}{2p-1}}$, we make use the classical $L^{2p}$ estimate for $\omega_{\mu}-$equation to get
\begin{equation*}
\Vert \omega_\mu(t)\Vert_{L^{2p}} \leq  \Vert {\omega}_\mu^0\Vert_{L^{2p}}+\Vert \nabla\theta_\mu\Vert_{L^1_tL^{2p}},
\end{equation*}
combined with the last estimate and {\bf(ii)} in Proposition \ref{prop v}, it happens  
\begin{equation*}\label{I-1-1} 
 \Vert v_\mu \Vert_{L^{\infty}_tL^{\infty}}  \leq  C_0(1+t)^{\frac{p-1}{2p-1}}    \Big(\Vert \omega_\mu^0 \Vert_{L^{2p}} ^{\frac{p}{2p-1}} +\Vert \nabla \theta_\mu  \Vert_{L^{1}_tL^{2p}}^{\frac{p}{2p-1}}\Big).
\end{equation*}
Substitute  the last estimate in  \eqref{I-1} we get
\begin{equation}\label{I-1-I} 
\mathrm{I}_1 \leq  C_0(1+t)^{\frac{p-1}{2p-1}}   \Big(\Vert \omega_\mu^0 \Vert_{L^{2p}} ^{\frac{p}{2p-1}} +\Vert \nabla \theta_\mu  \Vert_{L^{1}_tL^{2p}}^{\frac{p}{2p-1}}\Big)\|\theta_{\mu}\|_{L_t^{2\eta}L^{2p}}.
\end{equation}
To close our claim, it remains to bound the two terms $\Vert \nabla \theta_\mu  \Vert_{L^{1}_tL^{2p}}^{\frac{p}{2p-1}}$ and $\|\theta_{\mu}\|_{L_t^{2\eta}L^{2p}}$. For the first one,  H\"older's inequality with respect to time allows us to write
\begin{equation*}
\Vert \nabla \theta_\mu  \Vert_{L^{1}_tL^{2p}}^{\frac{p}{2p-1}}\le t^{\frac{2\eta-1}{2\eta}\frac{p}{2p-1}}\|\nabla\theta_\mu\|_{L^{2\eta}_t L^{2p}}^{\frac{p}{2p-1}},
\end{equation*}
whereas, for the second one is derived as follows
\begin{equation*}
\|\theta_{\mu}\|_{L_t^{2\eta}L^{2p}}\le t^{\frac{1}{2\eta}}\|\theta_{\mu}\|_{L_t^{\infty}L^{2p}}.
\end{equation*}
Gathering the last two estimates and plugging them in \eqref{I-1-I}, it holds 
\begin{equation*} 
\mathrm{I}_1 \leq  C_0(1+t)^{\frac{p-1}{2p-1} } t^{\frac{1}{2\eta}}  \Big(\Vert \omega_\mu^0 \Vert_{L^{2p}} ^{\frac{p}{2p-1}} +t^{\frac{ 2\eta-1}{2\eta} {\frac{p}{2p-1}} }   \Vert \nabla \theta_\mu  \Vert_{L^{2\eta}_tL^{2p}}^{\frac{p}{2p-1}}\Big)\|\theta_{\mu}\|_{L_t^{\infty}L^{2p}}.
\end{equation*}
Therefore 
\begin{eqnarray}\label{I1_1_1_1}
\mathrm{I}_1  &\leq &  C_0(1+t)^{\frac{p-1}{2p-1}+\frac{1}{2\eta}} (1+t)^{\frac{ 2\eta-1}{2\eta} {\frac{p}{2p-1}}} \Vert \nabla \theta_\mu  \Vert^{\frac{p}{2p-1}} _{L^{2\eta}_tL^{2p}} \|\theta_{\mu}\|_{L_t^{\infty}L^{2p}} \nonumber  \\ &\leq &  C_0(1+t)^{\frac{p-1}{2p-1}+\frac{1}{2\eta} +{\frac{ 2\eta-1}{2\eta}  {\frac{p}{2p-1}}} } \Vert \nabla \theta_\mu  \Vert^{\frac{p}{2p-1}} _{L^{2\eta}_tL^{2p}} \|\theta_{\mu}\|_{L_t^{\infty}L^{2p}}  .  
\end{eqnarray}  
On the one hand, by hypothesis $\theta^0_\mu\in B^{2-\frac2r}_{p,r}$, however, the assumption $\frac1p+\frac2r < 2$ leads to ${B_{{p,r}}^{2-\frac{2}{r}}}  \hookrightarrow L^{2p} $. Meaning that $\theta^0_\mu\in L^{2p}$. On the other hand, {\bf(i)}-Proposition \ref{prop v} ensures that $\|\theta_\mu\|_{L^\infty_t L^{2p}}$ is bounded, that is $\|\theta_\mu\|_{L^\infty_t L^{2p}}\le \|\theta_{\mu}^0\|_{L^{2p}}$. Finally,  by setting $r=\frac{2p-1}{p-1}, r'=\frac{2p-1}{p}$ and applying Young's inequality, one may write
\begin{equation}\label{I-1-f}
\mathrm{I}_1\le C_0\alpha^{\frac{2p-1}{p-1}}\Big(\frac{p-1}{2p-1}\Big)(1+t)^{ 1+\frac{1}{2\eta}\frac{2p-1}{p-1}+ \frac{ 2\eta-1}{2\eta}  \frac{p}{p-1}}  +\frac{1}{\alpha^{\frac{2p-1}{p}}}\Big(\frac{p}{2p-1}\Big)\Vert \nabla \theta_\mu  \Vert_{L^{2\eta}_t L^{2p}}.
\end{equation}
For $ \alpha = (\frac{4p}{2p-1})^{\frac{p}{2p-1}}$, it follows that 
\begin{eqnarray}\label{I-1-f}
\mathrm{I}_1&\le& C_0(\frac{4p}{2p-1})^{\frac{p}{p-1}}\Big(\frac{p-1}{2p-1}\Big)(1+t)^{1+\frac{1}{2\eta}\frac{2p-1}{p-1}+ \frac{ 2\eta-1}{2\eta}  \frac{p}{p-1}}  +\frac{1}{4}\Vert \nabla \theta_\mu  \Vert_{L^{2\eta}_t L^{2p}} \nonumber \\&\le& C_pC_0(1+t)^{1+\frac{1}{2\eta}\frac{2p-1}{p-1}+ \frac{ 2\eta-1}{2\eta}  \frac{p}{p-1}}  +\frac{1}{4}\Vert \nabla \theta_\mu  \Vert_{L^{2\eta}_t L^{2p}},
\end{eqnarray}
with $ C_p=(\frac{4p}{2p-1})^{\frac{p}{p-1}}\frac{p-1}{2p-1}$.

\hspace{0.5cm}Moving to estimate $\mathrm{I}_2$. Again Proposition \ref{op A} combined with assumption \eqref{Assup-Th-1} enables us to obtain
\begin{eqnarray}\label{I-2-f} 
\mathrm{I}_2\triangleq\bigg\| \int_0^t   \nabla^2 \mathbb{S}(t-\tau) ((\kappa(\theta_\mu)-1)\nabla \theta_\mu) d\tau \bigg\|_{L^{2\eta}_tL^{2p}} &\le& \Vert \kappa(\cdot )-1 \Vert_{L^\infty}  \Vert \nabla \theta_\mu  \Vert_{L^{2\eta}_tL^{2p}} \\
&\lesssim & \frac{1}{4}\Vert \nabla \theta_\mu  \Vert_{L^{2\eta}_tL^{2p}}.\nonumber
\end{eqnarray}
Finally, collecting \eqref{I-0-f}, \eqref{I-1-f} and \eqref{I-2-f} and plug them in \eqref{I-0+I-1+I-2} to infer that
\begin{equation*}
\Vert \nabla \theta_\mu  \Vert_{L^{2\eta}_tL^{2p}} \leq C_0(1+t)^{1+\frac{1}{2\eta}\frac{2p-1}{p-1}+ \frac{ 2\eta-1}{2\eta}  \frac{p}{p-1}} +\frac{1}{2}\Vert \nabla \theta_\mu  \Vert_{L^{2\eta}_tL^{2p}}.
\end{equation*}
Therefore
\begin{equation*}
\Vert \nabla \theta_\mu  \Vert_{L^{2\eta}_tL^{2p}} \le C_0(1+t)^{1+\frac{1}{2\eta}\frac{2p-1}{p-1}+ \frac{ 2\eta-1}{2\eta}  \frac{p}{p-1}}. 
\end{equation*} 
The function $(\eta,p)\longmapsto 1+\frac{1}{2\eta}\frac{2p-1}{p-1}+ \frac{ 2\eta-1}{2\eta}  \frac{p}{p-1} =\frac{2p-1}{p-1}+\frac{1}{2\eta}$ admits  $7/2$ as a maximum for   $(\eta,p)\in [1,\infty[\times]2,\infty[$  then we finally obtain  
\begin{equation}\label{grad-theta}
\Vert \nabla \theta_\mu  \Vert_{L^{2\eta}_tL^{2p}} \le C_0(1+t)^{\frac{7}{2}}. 
\end{equation} 
Now, we come back to estimate $\nabla^2 \theta_\mu$ in ${L^{\eta}_tL^{p}}$. For this aim, we develop the Duhamel formula \eqref{Duh-form} to write
\begin{eqnarray*}
\theta_\mu (t,x)&= & \mathbb{S}(t) \theta^0_\mu(x) -\int_0^t \mathbb{S}(t-\tau) v_\mu \cdot \nabla  \theta_\mu(\tau,x) d \tau  +\int_0^t  \mathbb{S}(t-\tau) \kappa'(\theta_\mu)(\nabla \theta_\mu )^2(\tau,x) d \tau \nonumber \\&+&\int_0^t  \mathbb{S}(t-\tau) ((\kappa(\theta_\mu)-1)\Delta \theta_\mu) (\tau,x) d \tau. 
\end{eqnarray*}
Apply $\nabla^2$ operator to this equation and take the $L^{\eta}_tL^{p}-$norm, it happens
\begin{eqnarray} \label{II-0+II-1+II-2}
\Vert \nabla ^2 \theta_\mu  \Vert_{L^{\eta}_tL^{p}} &\leq & \big\Vert \nabla ^2 \mathbb{S}(t)    \theta^0_\mu d\tau \big\Vert_{L^{\eta}_tL^{p}}+\bigg\Vert \int_0^t \nabla^2  \mathbb{S}(t-\tau) (v_  \mu  \cdot \nabla \theta_\mu) d\tau  \bigg\Vert_{L^{\eta}_tL^{p}} \nonumber \\&+&\bigg\Vert \int_0^t \nabla^2 \mathbb{S}(t-\tau)\kappa'(\theta_\mu)(\nabla \theta_\mu )^2 d\tau \bigg\Vert_{L^{\eta}_tL^{p}} + \bigg\Vert \int_0^t  \nabla ^2 \mathbb{S}(t-\tau) ((\kappa(\theta_\mu)-1)\Delta \theta_\mu) d\tau \bigg\Vert_{L^{\eta}_tL^{p}}\nonumber \\
&\triangleq &\mathrm{II} _0+\mathrm{II}_1+\mathrm{II}_2 +\mathrm{II}_3.
\end{eqnarray}  
To estimate $\mathrm{II}_0$, combine Proposition \ref{A 1} with the fact that $ \nabla^2 :{\dot{B}_{{p,\eta}}^{2-\frac{2}{\eta}}} \rightarrow   {\dot{B}_{{p,\eta}}^{\frac{-2}{\eta}}}  $ is a continuous map to write  
\begin{eqnarray*} 
\mathrm{II} _0 \triangleq\big\Vert \nabla ^2 \mathbb{S}(t)    \theta^0_\mu d\tau \big\Vert_{L^{\eta}_tL^{p}}&\lesssim &\Vert \nabla^2 \theta^0_\mu  \Vert_{\dot{B}_{{p,\eta}}^{\frac{-2}{\eta}}}.   
\end{eqnarray*}
The case where $ \frac{1}{p} +\frac{1}{r} \leq 1$ we have $r \geq \frac{p}{p-1}>1$. By taking $\eta =r $ we get
\begin{eqnarray*} 
\mathrm{II} _0 \triangleq\big\Vert \nabla ^2 \mathbb{S}(t)    \theta^0_\mu d\tau \big\Vert_{L^{\eta}_tL^{p}}&\lesssim &\Vert \nabla^2 \theta^0_\mu  \Vert_{\dot{B}_{{p,\eta}}^{\frac{-2}{\eta}}}     \nonumber \\ &\lesssim &   \Vert \theta^0_\mu  \Vert_{\dot B_{{p,\eta}}^{2-\frac{2}{\eta}}}    \\&\lesssim&   \Vert \theta^0_\mu  \Vert_{ B_{{p,\eta}}^{2-\frac{2}{\eta}}}  ,  
\end{eqnarray*} 
while, for $ \frac{1}{p} +\frac{1}{r}> 1$ we have $\frac{1}{\eta}=\frac1p+\frac2r-1$ which implies that $1<\eta<r$ and $ 2-\frac{2}{\eta} < 2-\frac{2}{r}$. Thus, {\bf{(i)}-}Proposition \ref{Emb} yields $ {B_{{p,\eta}}^{2-\frac{2}{r}}} \hookrightarrow {B_{{p,r}}^{1-\frac{2}{\eta}}} $ meaning that
\begin{equation}\label{II-0}
\mathrm{II}_0 \lesssim  \Vert \theta^0_\mu  \Vert_{B_{{p,r}}^{2-\frac{2}{r}}}.
\end{equation}
The term $\mathrm{II}_1$ is a consequence of Proposition  \eqref{op A}, H\"older inequality , that is
\begin{eqnarray*}
\mathrm{II}_1\triangleq \bigg\Vert \int_0^t \nabla^2  \mathbb{S}(t-\tau) (v_  \mu  \cdot \nabla \theta_\mu) d\tau  \bigg\Vert_{L^{\eta}_tL^{p}} \leq  C\Vert v_  \mu  \cdot \nabla \theta_\mu  \Vert_{L^{\eta}_tL^{p}} \leq   \Vert  v_  \mu  \Vert_{L^{2\eta}_tL^{2p}} \Vert  \nabla \theta_\mu  \Vert_{L^{2\eta}_tL^{2p}}. 
\end{eqnarray*}
For the term $\Vert  v_  \mu  \Vert_{L^{2\eta}_tL^{2p}}$, a particular Gagliardo-Nirenberg inequality and Proposition \ref{C.Z} leading to
\begin{equation*}
\Vert  v_  \mu  \Vert_{L^{2p}(\RR^2)}\le C\|v\|^{\frac1p}_{L^2(\RR^2)}\|\nabla v\|^{1-\frac1p}_{L^2(\RR^2)}\le C\|v\|^{\frac1p}_{L^2(\RR^2)}\|\omega\|^{1-\frac1p}_{L^2(\RR^2)}.
\end{equation*}
Putting together the last two estimates, it follows 
\begin{equation}\label{II*1}
\mathrm{II}_1 \lesssim  \|v\|^{\frac1p}_{L^\infty_t L^2(\RR^2)}\|\omega\|^{1-\frac1p}_{L^\infty_t L^2(\RR^2)}\Vert  \nabla \theta_\mu  \Vert_{L^{2\eta}_tL^{2p}}. 
\end{equation}
From {\bf(ii), {\bf(iii)}}-Proposition \ref{prop v}, we have 
\begin{equation*}
\|v\|^{\frac1p}_{L^2(\RR^2)} \leq  C_0(1+t)^{\frac{1}{p}}, \quad \|\omega\|^{1-\frac1p}_{L^2(\RR^2)} \leq C_0(1+t)^{1-\frac{1}{p} }.
\end{equation*}

Plugging the last estimate in \eqref{II*1}, so, in view of  \eqref{grad-theta} we deduce that
\begin{equation}\label{II-1}
\mathrm{II}_1 \leq  C_0(1+t)^{\frac{9}{2}}.
\end{equation}
Step by step Proposition \ref{op A} and using \eqref{grad-theta} yield
\begin{equation}\label{II-2}
\mathrm{II}_2  \leq  C \Vert \kappa'(\theta)\Vert_{L^\infty}  \Vert \nabla \theta_\mu  \Vert_{L^{2\eta}_tL^{2p}}^2 \leq  C_0(1+t)^{7} .
\end{equation}
and 
\begin{equation}\label{II-3}
 \mathrm{II}_3  \leq  \Vert \kappa(\cdot )-1 \Vert_{L^\infty}  \Vert \nabla^2 \theta_\mu  \Vert_{L^{\eta}_tL^{p}} \leq   \frac{1}{2} \Vert \nabla^2 \theta_\mu  \Vert_{L^{\eta}_tL^{p}}.
\end{equation}
At this stage, collecting \eqref{II-0} , \eqref{II-1} , \eqref{II-2}, and \eqref{II-3} and plug them in \eqref{II-0+II-1+II-2} we conclude that
\begin{equation}\label{grad 2 theta} 
\Vert \nabla^2 \theta_\mu \Vert_{L^{\eta}_tL^{p}} \le C_0(1+t)^{7} . 
\end{equation}
This ends the proof of Proposition \ref{prop theta}.
\end{proof}

As a consequence of the previous results, we shall control the quantity $\nabla \theta_\mu$ and $\omega_\mu$ in $L^1_t L^\infty$ and ${L^\infty_t L^\infty}$ space respectively. Especially, we will establish.
\begin{cor}\label{cor}
Let $(\omega_\mu,\theta _\mu )$ be a smooth solution of the system \eqref{Eq-5}. Then the following assertions are hold.
\begin{enumerate}
\item[{\bf(i)}] For every  $ t \geq 0 $ and $ p\in ]2,+\infty[$, we have
\begin{equation*}  
\Vert \nabla \theta_\mu \Vert_{L^{1}_t L^{\infty}}\leq  C_0(1+t)^{7}.
\end{equation*}
\item[{\bf(ii)}] For every  $ t \geq 0 $ and $ p\in ]2,+\infty[$ we have 
\begin{equation*}
\Vert \omega_\mu(t)\Vert_{L^\infty} \le C_0(1+t)^{7}.
\end{equation*} 
\end{enumerate}
\begin{proof}
{\bf(i)} Exploring the Gagliardo-Nirenberg inequality to obtain 
\begin{eqnarray*}
\Vert \nabla \theta_\mu (t)\Vert_{L^\infty} &\lesssim &  \Vert \nabla  \theta_\mu (t) \Vert_{L^{2}}^{\frac{p-2}{2p-2}} \Vert \nabla^2  \theta_\mu (t) \Vert_{L^{p}}^{\frac{p}{2p-2}},  
\end{eqnarray*}
so, Young's inequality leads  
\begin{eqnarray}\label{G-th}
\Vert \nabla \theta_\mu (t)\Vert_{L^\infty} &\le &  \Big({\frac{p-2}{2p-2}}\Big) \Vert \nabla  \theta_\mu (t) \Vert_{L^{2}} +\Big({\frac{p}{2p-2}}\Big)  \Vert \nabla^2  \theta_\mu (t) \Vert_{L^{p}}  \\ & \le & \nonumber \Vert \nabla  \theta_\mu (t) \Vert_{L^{2}}+  \Vert \nabla^2  \theta_\mu (t) \Vert_{L^{p}}.
\end{eqnarray}
Integrating in time over $[0,t]$ and make use the Cauchy-Schwarz inequality with respect to time we get 
\begin{eqnarray*}
\Vert \nabla \theta_\mu \Vert_{L^{1}_t L^\infty} &\le &    t^{\frac{1}{2}}\Vert \nabla  \theta_\mu  \Vert_{L^{2}_t L^{2}}+  \Vert \nabla^2  \theta_\mu  \Vert_{L^{1}_t L^{p}}.
\end{eqnarray*}
In particular, \eqref{grad 2 theta} for $\eta=1$ and \eqref{grad teta L^2} yield 
\begin{equation*}\label{3.83}
\Vert \nabla \theta_\mu \Vert_{L^{1}_t L^\infty} \le  t^{\frac{1}{2}} \kappa^{\frac{1}{2}}_0 \Vert \theta^0_\mu \Vert_{ L^{2}} + C_0(1+t)^{7} . 
\end{equation*}
Thus we have 
\begin{equation}\label{gred teta L infine}
\Vert \nabla \theta_\mu \Vert_{L^{1}_t L^\infty} \le    C_0(1+t)^{7} .
\end{equation}
{\bf(ii)} The maximum principal for $\omega_\mu-$equation enubles us to write 
\begin{equation*}
\Vert \omega_\mu(t)\Vert_{L^\infty} \leq  \Vert {\omega}_\mu^0\Vert_{L^\infty} +\Vert \nabla\theta_\mu\Vert_{L^1_tL^\infty},
\end{equation*}
combined with \eqref{gred teta L infine} gives the desired estimate, so, the proof of Corollary is completed.
\end{proof}
\end{cor}
\hspace{0.5cm}By exploring the previous preparatory part, in particular, the Proposition \ref{G-V}, we shall prove Theorem \ref{Th-1} in the more general case. More precisely, we will establish the following theorem.
\begin{Theo}\label{theo-GP}
Let $0<\EE<1, X_0$ be a family of admissible vector fields and $v_{\mu}^0$ be a free-divergence vector field in the sense that $\omega^0_{\mu}\in L^2\cap C^{\EE}(X_0)$. Let $  \theta_\mu^0 \in {L^{2}}\cap B_{p,r}^{2-\frac{2}{r}}$ with $(p,r)\in ]2, \infty[\times]1,\infty[$ be such that $\frac1p+\frac2r\le1 $, , then for $\mu\in]0,1[$ the system \eqref{Eq-3} admits a unique global solution 
\begin{equation*}
(v_\mu,\theta_\mu)\in L^\infty\big([0,T];\Lip\big)\times L^\infty\big([0,T];L^2\big)\cap L^{\eta}\big([0,T];W^{2,p}\big),\quad\eta >1. 
\end{equation*}
More precisely,
\begin{equation*}
\| \nabla v_\mu\|_{L^\infty_t L^{\infty}}\leq  C_0e^{C_0t^{8}}   . 
\end{equation*}
Furthermore,
\begin{equation*}
\| \omega_\mu\|_{L^\infty_tC^{\epsilon}(X_t)}+{\widetilde{\|}} X_{\lambda }\|_{L^\infty_t C^{\epsilon}(X_t)} +\|  \Psi_\mu \|_{{L^\infty_tC^\epsilon}(X_t)} \leq   C_0e^{\exp C_0t^{8}}.
\end{equation*}
\end{Theo} 
\begin{proof}
The existence part of the theorem is classical and can be done for example by using a standard recursive method, see, e.g. \cite{Paicu-Zhu}. We will focus on explicit that the velocity is a Lipschtizian function through the striated regularity of its vorticity. For this aim, taking the directional derivative $\partial_{X_t,\lambda}$ to $\omega_{\mu}-$equation in the system \eqref{Eq-5}, it follows in accordance with Proposition \ref{com} that
\begin{equation*}\label{Eq8}
(\partial_t+v\cdot\nabla-\mu\Delta)\partial_{X_t,\lambda}\omega_\mu={X_{t,\lambda}} \cdot \nabla \partial_1\theta_\mu-\mu[\Delta,{X_{t,\lambda}}]\omega_\mu.
\end{equation*}
The key thus is to estimate the commutator $\mu[\Delta,{X_{t,\lambda}}]\omega_\mu$. Via Bony's decomposition, we write
\begin{equation*}
\mu[\Delta,X_{t,\lambda}]\omega_\mu=\mathfrak{A}+\mu\mathfrak{B},
\end{equation*}
with
\begin{equation*}
\mathfrak{A}\triangleq2\mu T_{\nabla X^{i}_{t,\lambda}}\partial_i\nabla\omega_\mu+2\mu T_{\partial_i\nabla\omega_\mu} \nabla  X^{i}_{t,\lambda}+ \mu T_{\Delta  X^{i}_{t,\lambda}} \partial_i\omega_\mu+\mu T_{\partial_i\omega_\mu}\Delta  X{^{i}_{t,\lambda}}.
\end{equation*}
and
\begin{equation*}
\mathfrak{B} \triangleq 2\mathcal{R}(\nabla X^i_{t,\lambda},\partial_i\nabla \omega_\mu)+\mathcal{R}(\Delta X^i_{t,\lambda},\partial_i\omega_\mu).
\end{equation*}
The famous Theorem 3.38 page 162 in \cite{Bahouri-Chemin-Danchin} confirms us
\begin{equation}\label{X omega}
\Vert\partial_{X_{\lambda}} \omega_\mu \Vert_{L_t^{\infty}C^{\epsilon-1}}\leq C e^{CV_\mu(t)}\Big(\Vert\partial_{X_0,\lambda}\omega^0_\mu\Vert_{C^{\epsilon-1}}+\Vert\partial_{X_{\lambda}}\partial_1\theta_\mu \Vert_{L^1_tC^{\epsilon-1}}+(1+\mu t)\Vert \mathfrak{A} \Vert_{{L_t^\infty}C^{\epsilon-3}}+ \mu \Vert \mathfrak{B} \Vert_{\widetilde{L}_t^1C^{\epsilon-1}}\Big )\\.
\end{equation}
To simplify our presentation we set
\begin{eqnarray}
\mathrm{III_1} &\triangleq & \Vert\partial_{X_0,\lambda}\omega^0_\mu\Vert_{C^{\epsilon-1}}+\Vert\partial_{X_{\lambda}}\partial_1\theta_\mu \Vert_{L^1_tC^{\epsilon-1}} \nonumber \\ \mathrm{III_2} & \triangleq &  \Vert \mathfrak{A} \Vert_{{L_t^\infty}C^{\epsilon-3}} \nonumber \\ \mathrm{III_3} & \triangleq & \mu \Vert \mathfrak{B} \Vert_{\tilde{L}_t^1C^{\epsilon-1}} \nonumber .
\end{eqnarray}
We embark by estimating III$_1$. The fact that $L^p \hookrightarrow  C^{\epsilon-1} $ for $p>\frac{2}{1-\EE}$ and H\"older's inequality leading to 
\begin{eqnarray*}
\mathrm{III_1} &\leq & C \big( \Vert\partial_{X_0,\lambda}\omega^0\Vert_{L^p}+\Vert{X_{\lambda}} \cdot \nabla  \partial_1\theta\Vert_{L^1_tL^{p}}\big)\\ &\leq & C \big( \Vert\partial_{X_0,\lambda}\omega^0_\mu \Vert_{L^p}+\Vert{X_{\lambda}} \Vert_{L^\infty_t L^{\infty}} \Vert \nabla^2 \theta_\mu \Vert_{L^1_t L^{p}}\big).
\end{eqnarray*}
In particular, \eqref{grad 2 theta} for $ \eta =1$ and the embedding $ C^\epsilon \hookrightarrow L^\infty$ ensure that 
\begin{eqnarray}\label{I estimate of grad v}
\mathrm{III_1} &\leq & C_0  (1+t )^{7} \Vert X_{t,\lambda }\Vert_{L^{\infty}C^{\epsilon}}.
\end{eqnarray}
For the term $\mathrm{III_2}$, we inspire the idea from \cite{Bahouri-Chemin-Danchin,Hmidi-1} to state
\begin{equation*} 
\mathrm{III_2} \leq C \Vert\omega _\mu \Vert _{L^\infty_t L^\infty}  \Vert X_{t,\lambda }\Vert_{L^{\infty}C^{\epsilon}},
\end{equation*}
which provides in view of {\bf(ii)} in Corollary \ref{cor} to
\begin{equation}\label{II estimate of grad v}
\mathrm{III_2} \leq C_0(1+t )^{7}    \Vert X_{t,\lambda }\Vert_{C^{\epsilon}}.
\end{equation}
Let us move to estimate $\mathrm{III_3}$, by employing again \cite{Bahouri-Chemin-Danchin,Hmidi-1}, one obtains
\begin{equation}\label{III 1}
\mathrm{III_3} \leq C \mu \Vert\omega_\mu\Vert _{\widetilde{L}^1_tB^2_{\infty,\infty}}  \Vert X_{t,\lambda }\Vert_{C^{\epsilon}}.
\end{equation}
Concerning the term $\Vert\omega_\mu\Vert _{\widetilde{L}^1_tB^2_{\infty,\infty}}$, we make use the maximal regularity stated in Proposition \ref{prop2.10} for $a=\omega_\mu, g=\partial_1\theta_\mu, r=1, s=0$, and $p_1=p_2=\infty$ to ensure
\begin{equation*} 
\mu \Vert\omega_\mu\Vert _{\widetilde{L}^1_tB^2_{\infty,\infty}}   \leq Ce^{CV_\mu(t)}(1+\mu t) \Big( \Vert\omega^0_\mu\Vert _{B^0_{\infty,\infty}}+ \int_{0}^{t} \Vert \partial_1\theta_\mu(\tau) \Vert _{B^0_{\infty,\infty}} d\tau \Big ).
\end{equation*}
Or, the embedding $L^\infty \hookrightarrow B^0_{{\infty,\infty}}$ implies
\begin{equation*} 
\mu \Vert\omega_\mu\Vert _{\widetilde{L}^1_tB^2_{\infty,\infty}}  \leq Ce^{CV_\mu(t)}(1+\mu t) \big(\Vert\omega^0 _\mu \Vert _{L^{\infty}}+  \Vert \nabla \theta_\mu \Vert _{L^1_t L^{\infty}}\big).
\end{equation*}
By means of {\bf(i)} in Corollary \ref{cor} we get
\begin{equation*} 
 \mu \Vert\omega_\mu\Vert _{\widetilde{L}^1_tB^2_{\infty,\infty}}   \leq  C_0e^{CV_\mu(t)}(1+\mu t)(1+t )^{7},
\end{equation*}
combined with \eqref{III 1}, we end up with
\begin{equation} \label{III estimate of grad v}
  \mathrm{III_3}  \leq  C_0e^{CV_\mu(t)}(1+\mu t)(1+t )^{7}.
\end{equation}
Adding \eqref{I estimate of grad v},\eqref{II estimate of grad v},\eqref{III estimate of grad v} and put them in \eqref{X omega}, bearing in mind that $\mu \in ]0,1[ $ we infer that
\begin{equation}\label{Est11} 
\Vert\partial_{X_{t,\lambda}} \omega_\mu\Vert_{L_t^{\infty}C^{\epsilon-1}} \leq C_0e^{CV_\mu(t)}(1+t )^{8} \widetilde{\Vert} X_{\lambda,t} \Vert_{L^{\infty}C^{\epsilon}}.
\end{equation}
Now, we bound the term $\widetilde{\Vert }X_{\lambda,t} \Vert_{C^{\epsilon}}$. Thanks to Proposition \ref{prop2.9} for $a= X_{t,\lambda }, s=\epsilon$ and $p=r= \infty$, we readily get
\begin{equation}\label{X C epsilon -1}
\Vert X_{t,\lambda }\Vert_{C^{\epsilon}}\leq Ce^{CV_\mu(t)}\Big(\Vert X_{0,\lambda}\Vert_{C^{\epsilon}}+ \int_0^t e^{- CV_{\mu}(\tau)} \Vert  \partial_{X_{\tau,\lambda}} v_\mu(\tau)\Vert_{C^{\epsilon}} d\tau\Big).             
\end{equation}
We make use the following result which its proof can be found in \cite{Bahouri-Chemin-Danchin,Chemin2}
\begin{equation*} 
\Vert  \partial_{X_{t,\lambda}} v_\mu(t)\Vert_{C^{\epsilon}} \leq  C\big( \Vert \nabla v_\mu(t) \Vert_{L^\infty} \widetilde{\Vert}  X_{t,\lambda }\Vert_{C^{\epsilon}}+\Vert \omega_\mu(t) \Vert_{C^{\epsilon-1}}\big).
\end{equation*}
Thus we get in view the last estimate and \eqref{Est11},
\begin{equation}\label{X v}
\Vert  \partial_{X_{t,\lambda}} v_\mu(t)\Vert_{C^{\epsilon}} \leq  C \widetilde{\Vert}  X_{t,\lambda }\Vert_{C^{\epsilon}}\Big(\Vert \nabla v_\mu(t) \Vert_{L^\infty} + C_0e^{CV_\mu(t)}(1+t )^{8}\Big).
\end{equation}
Substituting \eqref{X v} in\eqref{X C epsilon -1} with an obvious change of constant we have

\begin{equation*}
{\Vert} X_{t,\lambda }\Vert_{C^{\epsilon}}\leq {Ce^{CV_\mu(t)}\bigg(\Vert X_{0,\lambda}\Vert_{C^{\epsilon}} + C_0\int_0^t e^{-{CV_{\mu}(\tau)}} \widetilde{\Vert}}  X_{\tau,\lambda }\Vert_{C^\epsilon}\Big( \Vert \nabla v_\mu(\tau) \Vert_{L^\infty} +(1+\tau )^{8}\Big)d\tau \bigg).  
\end{equation*}
To close our claim we treat the term ${\Vert} \Div X_{t,\lambda }\Vert_{C^{\epsilon}}$ by applying the divergence operator to \eqref{tr} and using the fact $\Div v_{\mu} =0$ we eventually get $(\partial_t+v\cdot\nabla) \Div {X_{t,\lambda}}=0$, so in view of Proposition \ref{prop2.9}, it happens  
\begin{equation}\label{Div-X}
\|\Div {X_{t,\lambda}}\|_{C^\EE}\le Ce^{CV_\mu(t)}\|\Div X_{0,\lambda}\|_{C^\EE}.
\end{equation}
Combining the last two estimates to conclude that
\begin{equation}\label{3.52}
e^{-CV_\mu(t)}\widetilde{\Vert} X_{t,\lambda }\Vert_{C^{\epsilon}}\le C\bigg(\widetilde{\Vert} X_{0,\lambda}\Vert_{C^{\epsilon}} + C_0\int_0^t e^{-{CV_{\mu}(\tau)}} \tilde{\Vert}  X_{\tau,\lambda }\Vert_{C^\epsilon}\Big( \Vert \nabla v_\mu(\tau) \Vert_{L^\infty}+(1+\tau)^{8}\Big)d\tau\bigg).
\end{equation}   
At this stage Gronwall's inequality tells us 
\begin{equation}\label{Est-X}
{\widetilde{\Vert}} X_{t,\lambda }\Vert_{C^{\epsilon}} \leq C_0 e^{ C_0  V_\mu(t)}e^{ C_0 t^{9}}.
\end{equation}
Gathering \eqref{Est11} and \eqref{Est-X}, one obtains
\begin{equation*}\label{3.26.}
\Vert\partial_{X_{t,\lambda}} \omega_\mu(t)\Vert_{C^{\epsilon-1}} \leq C_0e^{ C_0  V_\mu(t)}e^{ C_0 t^{9}}.
\end{equation*}
Moreover, from the last two estimates and {\bf(ii)} in Corollary \ref{cor} we thus get   
\begin{equation}\label{3.27}
\Vert\partial_{X_{t,\lambda}} \omega_\mu(t)\Vert_{C^{\epsilon-1}}+\Vert \omega_\mu(t)\Vert_{L^\infty} {\widetilde{\Vert}} X_{t,\lambda }\Vert_{C^{\epsilon}} \leq C_0 e^{ C_0  V_\mu(t)} e^{ C_0 t^{9}}. 
\end{equation}
Finally, we must estimate $\omega_\mu$ in anisoropic H\"older space $ C^{\epsilon}(X_t) $. For this goal, we recall that 
\begin{equation}\label{3.56}
\Vert \omega \Vert_{C^{\epsilon}(X)}\triangleq \frac{1}{I(X_t)}\bigg( \Vert \omega\Vert_{L^{\infty}}\sup_{\lambda \in \Lambda}\widetilde{\Vert}X_\lambda\Vert_{C^{\epsilon}}+ \sup_{\lambda \in \Lambda }\Vert \partial_{X_{\lambda}} \omega \Vert_{C^{\epsilon-1}}\bigg).
\end{equation}
To control the term  $I(X_t)$ we apply the derivative in time to the quantitity $\partial_{X_{0,\lambda}}\Psi$, it follows  
\begin{equation*}
\left\{ \begin{array}{ll}
\partial_t \partial_{X_0,\lambda} \Psi(t,x) = \nabla v(t,\Psi(t,x)) \partial_{X_0,\lambda} \psi(t,x) &\\
\partial_{X_{0,\lambda}} \Psi(0,x) ={X_{0,\lambda}}.
\end{array} 
\right.
\end{equation*}
The time reversibilty of the previous equation and Gronwall's inequality ensure that 
\begin{equation*}
| X_{0,\lambda}(x)| \leq |\partial_{X_{0,\lambda}}\Psi(t,x)|e^{V_{\mu}(t)}.
\end{equation*}   
From {\bf(ii)} in  Definition \ref{3.1} we deduce  that 
\begin{equation}\label{I-X-t}
 I(X_t)\geq I(X_0)e^{-V_{\mu}(t)} > 0.
\end{equation}  
In accordance with \eqref{3.27}, \eqref{3.56} and \eqref{I-X-t}, we end up with 
\begin{equation}\label{Est15}
\Vert \omega_\mu(t)\Vert_{C^{\epsilon}(X_t)} \leq C_0 e^{ C_0 t^{9}}e^{ C_0  V_\mu(t)} .
\end{equation}
In keeping with logarithmic estimate stated in Theorem \ref{The estimate log}, one obtains in virtue of {\bf{(iii)-}} in Proposition \ref{prop v} and {\bf(ii)} in Corollary \ref{cor} the following
$$ \Vert \nabla v_\mu(t)\Vert_{L^{\infty}}\leq C\left( C_0(1+t) + C_0(1+t)^{7}  \log \left( e+ \frac{\Vert \omega_\mu(t) \Vert_{C^{\epsilon}(X)}}{\Vert \omega_\mu(t) \Vert_{L^{\infty}}} \right) \right).  $$
The property of increasing function $ x\longmapsto \log (e+\frac{a}{x})$ yields
$$ \Vert \nabla v_\mu(t)\Vert_{L^{\infty}}\leq  C_0 \left( (1+t)^{7} \log \left( e+ \frac{\Vert \omega_\mu(t) \Vert_{C^{\epsilon}(X)}}{\Vert \omega^0_\mu \Vert_{L^{\infty}}} \right) \right).  $$ 
It follows from \eqref{Est15} that
$$ \Vert \nabla v_\mu(t)\Vert_{L^{\infty}}\leq C_0(1+t)^{7}  \Big( (1+t)^{9}+\int_0^t  \Vert \nabla v_\mu(\tau)\Vert_{L^{\infty}}d\tau\Big). $$
Again, Gronwall's inequality gives 
\begin{equation}\label{Lip v}
\Vert \nabla v_\mu(t)\Vert_{L^{\infty}}\leq C_0e^{C_0t^{8} }.
\end{equation}
together with \eqref{Est15}, we may write   
\begin{equation}\label{Est16}
\Vert \omega_\mu(t)\Vert_{C^{\epsilon}(X_t)} \leq C_0 e^{ \exp C_0t^{8}} .
\end{equation} 
Now, it remains to bound $\Psi_\mu$ in $C^\epsilon(X_t)$. First we recall that $\partial_{X_{0,\lambda}} \Psi_\mu(t) = {X_{t,\lambda}}\circ\Psi_\mu(t) $, so we exploit in general case the definition $\|f\|_{C^\EE}=\|f\|_{L^\infty}+\sup_{x\ne y}\frac{|f(x)-f(y)|}{|x-y|^{\EE}}$, one has
\begin{equation*}
 \Vert  {X_{t,\lambda}}\circ\Psi_\mu(t) \Vert_{C^\epsilon} \leq  \Vert  {X_{t,\lambda}}\Vert _{C^\epsilon}{\Vert \nabla \Psi_\mu(t) \Vert}^{\EE}_{L^{\infty}} \leq  \Vert  {X_{t,\lambda}}\Vert _{C^\epsilon} e^{CV_{\mu}(t)},
\end{equation*}
where, we have used ${\Vert \nabla \Psi_\mu(t) \Vert}_{L^{\infty}} \leq e^{CV_{\mu}(t)}$. Consequently,  
\begin{equation}\label{3.31}
\Vert  \Psi_\mu(t) \Vert_{{C^\epsilon}(X_{t})} \leq C_0 e^{ \exp C_0t^{8}}.
\end{equation}
The proof of Theorem \ref{theo-GP} is finished. 
\end{proof}
\subsection*{Proof of Theorem \ref{Th-1}} The proof of Theorem \ref{Th-1} comes from Theorem \ref{theo-GP}, however it remains only to establish the persistence regularity of the boundary of the transported patch $\Omega_t$. For this reason, we will erect an initial admissible family $ X_0=(X_{0,\lambda})_{\lambda\in \{0,1 \} }$ for which $ \omega^0_\mu=1_{\Omega_0}\in C^{\EE}(X_{0})$. Since $\partial\Omega_0$ is a Jordan curve with $C^{1+\EE}-$regularity, so in light of Definition \ref{Char} there exists a real function $ f_0 \in C^{1+\epsilon}$ and neighborhood $W_0$ fill the following property 
\begin{equation*}
\partial \Omega _{0}=W_0 \cap f^{-1}(\{0\}),\quad\nabla f_0\ne 0\;\mbox{on}\; W_0.
\end{equation*}
Let $ \chi \in \mathcal{D}(\mathbb{R}^{2}) $ be such that $ 0\leq \chi \leq1, \supp \chi \subset W_0 $ and $\chi\equiv1$ on a small neighborhood of $W_1 \Subset W_0  $, then define the two vectors
\begin{displaymath} 
 X_{0,0} =\nabla^{\bot}f_0 \quad \textnormal{ and} \quad  X_{0,1} =(1-\chi) \binom {1}{0}. 
\end{displaymath}
First, we check easily that $ X_0=(X_{0,\lambda})_{\lambda \in \{0,1 \} } $ is non-degenerate in accoordance with {\bf(i)} in Definition \ref{3.1}, moreover, for $\lambda\in\{0,1\}$ we claim that  $X_{0,\lambda},\Div X_{0,\lambda}\in C^{\epsilon}$, so due to the Definition \ref{def:3.2.1} we deduce that $ X_0=(X_{0,\lambda})_{\lambda \in \{0,1 \} } $ is an admissible family. Second, $ X_{0}=(X_{0,\lambda})_{\lambda \in \{0,1 \} } $ is tangential family with respect to $\Sigma=\partial\Omega_0$. Indeed, we remark that

\begin{equation*}
\left\{\begin{array}{ll}
 X_{0,0}\cdot\nabla f_{{0}{|\partial\Omega_0}}= \nabla^{\bot} f_0\cdot f_{{0}_{|\partial\Omega_0}}=0 ,\\ X_{0,1}\cdot \nabla f_{{0}{|\partial\Omega_0}}=(1-\chi)\partial_x f_0=0.
\end{array}
\right.
\end{equation*}
In last line we have used the fact that $\chi\equiv 1$ on $W_1$. We thus deduce that $X_0\in\mathcal{T}^{\EE}_{\Sigma}$. 

\hspace{0.5cm}The Definition \ref{def3.2} ensures that the pushforward of the family vector field $X_0$ is defined for $x\in\RR^2$ and $t\ge0$ by $ X_{t,\lambda}(x)= (\partial_{X_{0,\lambda}}\Psi_\mu)(t,\Psi_\mu^{-1})(t,x) $. In light of \eqref{Div-X}, \eqref{Est-X}  and \eqref{I-X-t}, the family $X_t$ still remains non-degenerate and satisfies for $\lambda\in\{0,1\}$ the regurality $X_{t,\lambda},\Div X_{t,\lambda} \in C^\epsilon$ which helps us  to say that $ X_t=(X_{t,\lambda})_{\lambda \in \{0,1 \} }$ is also an admissible family.

Next, for $x_0\in\partial\Omega_0$ we parametrize the curve $\partial\Omega_0$ as a solution of the following ordinary equation

\begin{equation*}
\left\{
\begin{array}{ll}
\partial_\zeta\gamma^{0}(\zeta) =X_{0,0}\gamma^{0}(\zeta), &\\
\gamma^{0}(0) =x_0.
\end{array}
\right.
\end{equation*}
A straightforward computation yields $\gamma^{0} \in C^{1+\epsilon}(\RR,\RR^2)$. On the other hand, the legitimite way to define the evolution parametrization of $ \partial \Omega_t$ is the transport process, that is for all $t\geq0$ we set $$\gamma(t,\zeta)\triangleq \Psi(t,\gamma^{0}(\zeta)).$$ 

The criterion differentiation with respect to $\zeta$ leads $\partial_{\zeta}\gamma(t,\zeta)=( \partial_{X_{0,0}} \Psi_\mu)(t,\gamma^{0}(\zeta))$. But $\partial_{X_{0,0}} \Psi_\mu \triangleq {X_{0,0}}\circ\Psi_\mu $, so, in view of \eqref{3.31} one finds that $\partial_{X_{0,0}} \Psi _\mu\in L^{\infty}_{\loc}((\mathbb{R}_+,C^{\epsilon}) $. Finally, we infer that $\gamma(t) \in  L^{\infty}_{\loc}((\mathbb{R_+},C^{1+\epsilon})  $ this confirms the regularity persistence of the curve $ \partial \Omega_t$ through the time, so the proof of Theorem \ref{Th-1} is accomplished.

\section{Inviscid limit}
This section concerns the inviscid limit in the context of a smooth patch of the system \eqref{Eq-3} towards \eqref{Eq-4} whenever the viscosity parameter $\mu$ goes to $0$ and quantify the rate of convergence between velocities, densities, and the associated flows. For this reason, we embark on the following technical results. The first one deals with the regularity of the vortex patch in certain homogeneous Besov space where its proof can be found in detail in \cite{Meddour-Zerguine}. The second one cares about the complex interpolation between the Laplacian of the velocity and its vorticity by means of Biot-Savart law. Especially, we have.      
\begin{prop}Let $0<\epsilon<1$ and  $\Omega_0$ be a $C^{\epsilon+1}$ -bounded domain, then the function $\mathbf{1}_{\Omega_0}$ belongs to ${\dot{B}_{{p,\infty}}^{\frac{1}{p}}}.$
\end{prop}
\begin{prop}\label{Delta v}
Let$ (p,r,\beta) \in [1,+\infty]\times ]-1,1[ $ and $ v_\mu$ be free-divergence vector field covered by the Biot-Savart law $v_\mu =\Delta ^{-1} \nabla ^\bot \omega_\mu$ then the following estimate holds true.
$$ \Vert \Delta v_\mu  \Vert_{L^r_tL^p} \leq C \Vert  \omega_\mu  \Vert_{\widetilde{L} ^r_t {B_{{p,\infty}}^{\beta}}}^\frac{1+\beta}{2}  \Vert  \omega_\mu  \Vert_{\widetilde{L} ^r_t {B_{{p,\infty}}^{2+\beta}}}^\frac{1-\beta}{2}.$$
\end{prop}
\begin{proof} In order to establish this estimate, let $N\in\NN$ be an integer number that will be chosen later. We combine Biot-Savat law $\Delta v_{\mu}=\nabla^{\perp}\omega_{\mu}$ with interpolation in frequency and Bernstein inequality to write

\begin{eqnarray}\label{Ayat11}
\|\Delta v_{\mu}\|_{L^{r}_{t} L^p}&\le& \sum_{q\le N}\Vert\Delta_q\nabla^{\perp}\omega_{\mu}\Vert_{L^{r}_t L^p}+\sum_{q> N}\Vert\Delta_q\nabla^{\perp}\omega_{\mu}\Vert_{L^{r}_t L^p}\\
\nonumber&\le &\sum_{q\le N}2^{q(1-\beta)}2^{q\beta}\Vert\Delta_q\omega_{\mu}\Vert_{L^{r}_{t} L^p}+\sum_{q>N}2^{q(-1-\beta)}2^{q(2+\beta)}\Vert\Delta_{q}\omega_{\mu}\Vert_{L^{r}_{t} L^p}\\
\nonumber&\le & 2^{N(1-\beta)}\Vert\omega_{\mu}\Vert_{\widetilde L^{r}_t B^{\beta}_{p,\infty}}+2^{-N(1+\beta)}\Vert\omega_{\mu}\Vert_{\widetilde L^{r}_t B^{2+\beta}_{p,\infty}}.
\end{eqnarray}
Chosing $N$ be such that
\begin{equation*}
2^{N(1-\beta)}\Vert\omega_{\mu}\Vert_{\widetilde{L}^{r}_t B^{\beta}_{p,\infty}}\approx 2^{-N(1+\beta)}\Vert\omega_{\mu}\Vert_{\widetilde{L}^{r}_t B^{2+\beta}_{p,\infty}},
\end{equation*}
Hence
\begin{equation}\label{Ayat111}
2^{2N}\approx\frac{\Vert\omega_{\mu}\Vert_{\widetilde{L}^{r}_t B^{2+\beta}_{p,\infty}}}{\Vert\omega_{\mu}\Vert_{\widetilde{L}^{r}_t B^{\beta}_{p,\infty}}}.
\end{equation}
Plugging \eqref{Ayat111} in \eqref{Ayat11} to obtain the desired estimate, so the proof is completed.
\end{proof}
\hspace{0.5cm}At this stage, we state the general version of Theorem \eqref{Theo1.2}. More precisely, we will prove the following theorem. 
\begin{Theo}\label{Theo4.2}
Let $ (v_\mu,\theta_\mu)$ and $ (v,\theta)$ be the solution of the \eqref{Eq-3} and \eqref{Eq-4} respectively with  $ (v_\mu^0,\theta_\mu^0)$ and $ (v^0,\theta^0)$ their initial data which satisfy the conditions of Theorem \ref{Th-1}. Assume that  $(\omega_\mu^0,\omega^0) \in L^\infty \cap {\dot{B}_{{p,\infty}}^{\frac{1}{p}}}\times L^2 \cap L^\infty $ then we have
\begin{eqnarray*}
 \Pi(t) &\leq & Ce^{t+ V_\mu (t)+V(t)+\Vert \nabla \theta  {\Vert}_{ L^1_tL^\infty} +\Vert \nabla \theta  {\Vert}_{ L^2_tL^\infty} (\Vert \nabla \theta  {\Vert}_{L^2_t L^p}+ \Vert \nabla \theta _\mu {\Vert}_{L^2_t L^p}) }\Big( \Pi(0) \\&& +C_0 (1+t)^{7}  (\mu t)^{\frac{1}{2}+\frac{1}{p}} (1+\mu t) \Big),
\end{eqnarray*}
with $\Pi(t)=\Vert v_\mu(t) -v (t)\Vert_{L^p} + \Vert \theta_\mu (t)-\theta(t) \Vert_{L^p}$ and 
\begin{equation*}
V(t)=\int_0^t \Vert \nabla v(\tau)\Vert_{L^\infty}d\tau, \quad V_\mu (t)=\int_0^t \Vert \nabla v_\mu(\tau)\Vert_{L^\infty}d\tau.
\end{equation*}
\end{Theo}     
\begin{proof}
Taking the difference between \eqref{Eq-3} and \eqref{Eq-4}, by setting $ U=v_\mu -v, \Theta =\theta_\mu -\theta $ and $P =p_\mu-p$ we find out that the triplet $(U,\Theta,P)$ gouverns the following evolution system. 
\begin{displaymath}\label{D-mu}
\left\{ \begin{array}{ll} 
\partial_{t} U+v_\mu \cdot\nabla U-\mu\Delta v_\mu =-\nabla P+\Theta \vec{e}_2 -U \cdot\nabla  v,  & \\
\partial_{t}\Theta +v_\mu\cdot\nabla \Theta -\nabla\cdot  \big(\kappa(\theta_\mu) \nabla \Theta  \big)=-U \cdot\nabla  \theta + \nabla\cdot \big(( \kappa(\theta_\mu)-\kappa(\theta ))\nabla \theta)\big),  & \\ 
\nabla\cdot U =0, &\\ 
({U},{\Theta})_{| t=0}=({U}_0,{\Theta}_0).\tag{D$_\mu$} 
\end{array} \right .
\end{displaymath}
Multiplying the first equation in the system \eqref{D-mu} by $ U|U|^{p-2} $ and integrating by part over $\RR^2$, bearing in mind that $\nabla \cdot v_{\mu}=\Div v=0 $ and H\"older's inequality ensure that
\begin{equation*} 
\frac{d}{dt} \Vert U(t) {\Vert}_{L^p}  \leq \Vert \nabla P(t) \Vert_{L^p} + \mu \Vert \Delta v_\mu (t) \Vert_{L^p} +\Vert \nabla v (t) \Vert_{L^\infty}\Vert U(t) {\Vert}_{L^p}  + \Vert \Theta(t) \Vert_{L^p}.
\end{equation*}
Integrating in time this differential inequality over $[0,t]$, one obtains 
\begin{eqnarray} \label{U}
 \Vert U(t) {\Vert}_{L^p} &\lesssim & \Vert U^0 {\Vert}_{L^p} + \int_0^t \Vert \nabla P(\tau) \Vert_{L^p}   d\tau +\mu \int_0^t \Vert \Delta v_\mu (\tau) \Vert_{L^p}  d\tau \nonumber\\ 
&+&\int_0^t  \Vert \nabla v (\tau) \Vert_{L^\infty}\Vert U(\tau) {\Vert}_{L^p} d\tau + \int_0^t  \Vert \Theta(\tau) \Vert_{L^p} d\tau 
\end{eqnarray}
On the other hand, we have $  -\nabla P=\partial_t U + v_\mu \cdot\nabla U +U \cdot\nabla v-\mu \Delta v_\mu  -\Theta \vec{e}_2$, thus by taking the divergence operator 

$$ -\Delta P= \nabla\cdot (v_\mu \cdot\nabla U + U \cdot\nabla v)) -\partial_2 \Theta   .$$
Because $\nabla\cdot v_\mu =\nabla\cdot v=0$ which implies that $\nabla\cdot(v_\mu \cdot\nabla U)=\nabla\cdot(U \cdot\nabla v_\mu)$ this yields 
$$ -\Delta P= \nabla\cdot (U \cdot(v_\mu + \nabla v)) -\partial_2 \Theta   .$$
Applying the operator $\nabla \Delta^{-1}$,then after a straightforward computation it holds 
\begin{equation*}
 -\nabla  P = \nabla \Delta^{-1}  \nabla\cdot (U\cdot \nabla (v_\mu + v)) +\nabla \Delta^{-1} \partial_2 \Theta.
\end{equation*} 
The quantity $ -\nabla  P$ can be seen as a Riesz transform. Kowning that this latter maps $L^p$ into itself for $p \in[2,+\infty[$, hence we find
\begin{equation*}
\Vert \nabla P(\tau) \Vert_{L^p} \lesssim \Vert U(\tau) \Vert_{L^p} (\Vert \nabla  v_{\mu}(\tau) \Vert_{L^\infty}+\Vert \nabla v(\tau) \Vert_{L^\infty}) + \Vert \Theta( \tau) \Vert_{L^p}. 
\end{equation*}
Plugging the last estimate into \eqref{U} we shall have 
\begin{eqnarray}\label{U 1 }
\Vert U(t) {\Vert}_{L^p} & \lesssim & \Vert U^0 {\Vert}_{L^p}+ \int_0^t \Vert U(\tau) \Vert_{L^p} (\Vert \nabla  v_{\mu}(\tau) \Vert_{L^\infty}+\Vert \nabla v(\tau) \Vert_{L^\infty}) d\tau \nonumber\\ &+& \int_0^t  \Vert \Theta(\tau) \Vert_{L^p} d\tau +\mu\int_0^t \Vert \Delta v_\mu (\tau) \Vert_{L^p}  d\tau.
\end{eqnarray} 
By the same fashion for $\Theta-$equation we write
\begin{eqnarray*}
  \frac{d}{dt} \Vert \Theta(t) {\Vert}^p_{L^p}+\int_{\mathbb{R}^2} \kappa(\theta_\mu)\nabla \Theta  \nabla(|\Theta| ^{p-2} \Theta)   dx &=&-\int_{\mathbb{R}^2} U\cdot\nabla \theta |\Theta| ^{p-2} \Theta dx \nonumber\\&-& \int_{\mathbb{R}^2} ( \kappa(\theta_\mu)-\kappa(\theta ))\nabla \theta   \nabla(|\Theta| ^{p-2} \Theta) dx.      
\end{eqnarray*}
Thanks to \eqref{Eq-2}, we get 
\begin{eqnarray}\label{kappa}
\frac{d}{dt} \Vert \Theta(t) {\Vert}^p_{L^p}  &+& \kappa_0^{-1} (p-1) \int_{\mathbb{R}^2} |\nabla \Theta|^2 |\Theta| ^{p-2}    dx \leq \int_{\mathbb{R}^2} |U\cdot \nabla \theta ||\Theta| ^{p-1}  dx      \nonumber\\  &+ & (p-1)  \int_{\mathbb{R}^2} | \kappa(\theta_\mu)-\kappa(\theta )| | \nabla \theta  \cdot \nabla \Theta ||\Theta| ^{p-2}  dx.\nonumber\\
\end{eqnarray}
On the other hand, Taylor's formula yields  
\begin{eqnarray*}
\kappa(\theta_\mu)-\kappa(\theta)& =& (\theta_\mu -\theta )\int_{0}^{1}\kappa'(\theta_\mu +\vartheta(\theta-\theta_\mu))d\vartheta\\
&=&\Theta\int_{0}^{1}\kappa'(\theta_\mu +\vartheta(\theta-\theta_\mu))d\vartheta.
\end{eqnarray*}
Thus \eqref{Eq-2} leads to
\begin{equation*}
|\kappa(\theta_\mu)-\kappa(\theta )| \le \kappa_{0}|\Theta|.  
\end{equation*}
Plug the last estimate in \eqref{kappa}, then it follows
\begin{eqnarray}
\frac{d}{dt} \Vert \Theta(t) {\Vert}^p_{L^p}  + \kappa_0^{-1} (p-1) \int_{\mathbb{R}^2} |\nabla \Theta|^2 |\Theta| ^{p-2}    dx &\leq & \int_{\mathbb{R}^2} |U\cdot \nabla \theta ||\Theta| ^{p-1}  dx      \nonumber\\  &+ & \kappa_0 \int_{\mathbb{R}^2}  |  \nabla \theta  \cdot \nabla \Theta ||\Theta| ^{p-1}  dx .   
\end{eqnarray}
Thus H\"older's inequality ensures that 
\begin{equation*}
\frac{d}{dt} \Vert \Theta(t) {\Vert}_{L^p} \leq  \Vert \nabla \theta (t) {\Vert}_{L^\infty}\Vert  U(t) {\Vert}_{L^p}+\kappa_0 \Vert \nabla \theta (t) {\Vert}_{L^\infty}  \Vert \nabla \Theta(t) {\Vert}_{L^p}.
\end{equation*}
Integrating in time over $[0,t]$, it follows 
\begin{equation}\label{T}
\Vert \Theta(t) {\Vert}_{L^p}  \leq\Vert \Theta^0 {\Vert}_{L^p} + \kappa_0  \int_0^t \Vert \nabla \theta (\tau) {\Vert}_{L^\infty}  \Vert \nabla \Theta(\tau) {\Vert}_{L^p}d\tau +\int_0^t \Vert \nabla \theta (\tau) {\Vert}_{L^\infty}   \Vert  U(\tau) {\Vert}_{L^p} d\tau .
 \end{equation} 
Gathering \eqref{U 1 } and \eqref{T}, one has 
\begin{eqnarray}
\Pi(t)  &\lesssim & \Pi(0) + \int_0^t  \Pi(\tau) \Big(1 +\Vert \nabla \theta (\tau) {\Vert}_{L^\infty}  +\Vert \nabla  v_{\mu}(\tau) \Vert_{L^\infty}+\Vert \nabla v(\tau) \Vert_{L^\infty} \Big)d\tau  \nonumber\\  &+&  \int_0^t\Vert \nabla \theta (\tau) {\Vert}_{L^\infty}  \big(\Vert \nabla \theta( \tau) {\Vert}_{L^p}+ \Vert \nabla \theta _\mu(\tau) {\Vert}_{L^p}\big) d\tau +\mu \int_0^t \Vert \Delta v_\mu (\tau) \Vert_{L^p}  d\tau,
\end{eqnarray}
with $\Pi(t)=\Vert v_\mu(t) -v (t)\Vert_{L^p} + \Vert \theta_\mu (t)-\theta(t) \Vert_{L^p } $, so Gronwall's inequality gives 
\begin{equation*}
\Pi(t) \leq Ce^{t+ V_\mu (t)+V(t)+\Vert \nabla \theta  {\Vert}_{ L^1_tL^\infty} +\int_0^t\Vert \nabla \theta (\tau) {\Vert}_{L^\infty} \big(\Vert \nabla \theta(\tau) {\Vert}_{L^p}+ \Vert \nabla \theta _\mu(\tau) {\Vert}_{L^p}\big) d\tau}\Big(\Pi(0) + \mu \Vert \Delta v_\mu \Vert_{ L^1_tL^p}\Big).
 \end{equation*}
Again, H\"older's inequality in time variable provides us
\begin{equation}\label{gamma}
\Pi(t) \leq Ce^{t+ V_\mu (t)+V(t)+\Vert \nabla \theta  {\Vert}_{ L^1_tL^\infty} +\Vert \nabla \theta  {\Vert}_{ L^2_tL^\infty} \big(\Vert \nabla \theta  {\Vert}_{L^2_t L^p}+ \Vert \nabla \theta _\mu {\Vert}_{L^2_t L^p}\big) }\Big(\Pi(0) + \mu \Vert \Delta v_\mu \Vert_{ L^1_tL^p}\Big).
 \end{equation}
Now, let us move to bound the term $ \mu \Vert \Delta v_\mu \Vert_{ L^1_tL^p} $ which considered as the source of the rate of convergence. To do this, Proposition \ref{Delta v} for  $\beta=\frac{1}{p}$ and $r=1$ yields
\begin{equation}\label{Delta v}
\mu \Vert \Delta v_\mu \Vert_{ L^1_tL^p} \leq \mu \Vert  \omega_\mu  \Vert_{\widetilde{L} ^1_t {B_{{p,\infty}}^\frac{1}{p}}}^{\frac{1}{2}+\frac{1}{2p}}\Vert  \omega_\mu  \Vert_{\widetilde{L} ^1_t {B_{{p,\infty}}^{2+\frac{1}{p}}}}^{\frac{1}{2}-\frac{1}{2p}}.
\end{equation}
Setting 
\begin{equation*}
\textnormal{IV}_1=\mu \Vert  \omega_\mu  \Vert_{\widetilde{L} ^1_t {B_{{p,\infty}}^\frac{1}{p}}}^{\frac{1}{2}+\frac{1}{2p}}, \quad\textnormal{IV}_2=\Vert  \omega_\mu  \Vert_{\widetilde{L} ^1_t {B_{{p,\infty}}^{2+\frac{1}{p}}}}^{\frac{1}{2}-\frac{1}{2p}}.
\end{equation*}
H\"older's inequality and Proposition \ref{prop2.10} confirm us 
\begin{eqnarray}\label{w I-1}
 \textnormal{IV}_1 &\leq &\mu t^{\frac{1}{2}+\frac{1}{2p}}  \Vert  \omega_\mu  \Vert_{\widetilde{L} ^\infty_t {B_{{p,\infty}}^\frac{1}{p}}}^{\frac{1}{2}+\frac{1}{2p}} \nonumber\\  &\leq & C  
 e^{CV_\mu(t)}  \mu t^{\frac{1}{2}+\frac{1}{2p}} \Big( \Vert  \omega_\mu ^0 \Vert_{ {B_{{p,\infty}}^\frac{1}{p}}} +  \Vert \nabla \theta_\mu  \Vert_{L ^1_t {B_{{p,\infty}}^\frac{1}{p}}} \Big)^{\frac{1}{2}+\frac{1}{2p}}.
\end{eqnarray}
But, in view of the continuity $\nabla: B_{{p,\infty}}^{\frac{1}{p}+1}\rightarrow B_{{p,\infty}}^\frac{1}{p}$,  we infer that
\begin{equation}\label{w I1}
  \Vert \nabla \theta_\mu  \Vert_{L ^1_t {B_{{p,\infty}}^\frac{1}{p}}}  \leq   \Vert \theta_\mu  \Vert_{L  ^1_t {B_{{p,\infty}}^{\frac{1}{p}+1}}}. 
\end{equation}
In accordance with Proposition \ref{Emb}, we have $W^{2,p} \hookrightarrow   B_{{p,\infty}}^{\frac{1}{p}+1}$. Consequently,
\begin{equation}\label{Eq-MZ0}
\Vert \theta_\mu  \Vert_{L ^1_t {B_{{p,\infty}}^{\frac{1}{p}+1}}}   \lesssim  \Vert \theta_\mu  \Vert_{L ^1_t W^{2,p}}.
\end{equation}
Let us bound $\Vert \theta_\mu  \Vert_{L ^1_t W^{2,p}}\triangleq\Vert \theta_\mu  \Vert_{L ^1_t L^{p}}+\Vert \nabla\theta_\mu  \Vert_{L ^1_t L^{p}}+\Vert \nabla^2\theta_\mu  \Vert_{L ^1_t L^{p}}$. For the first and third terms, it is enough to apply {\bf(i)} in Proposition \ref{prop v} and Proposition \ref{prop theta} to get 
\begin{equation}\label{Eq-MZ}
\Vert \theta_\mu  \Vert_{L ^1_t L^{p}}\le t\Vert \theta_\mu^{0}  \Vert_{L^{p}},\quad \Vert \nabla^2\theta_\mu  \Vert_{L ^1_t L^{p}}\le C_0(1+t)^{7}.
\end{equation}
For the second term $\Vert \nabla\theta_\mu  \Vert_{L ^1_t L^{p}}$, employ Gagliardo-Nirenberg's and Young's inequalities, it follows   
\begin{eqnarray}\label{Eq-MZ1}
\Vert \nabla \theta_\mu(t) \Vert_{L^{p}} &\le& \Vert \nabla \theta_\mu(t) \Vert_{L^{2}}^{\frac{2}{p}}\Vert \nabla \theta_\mu(t) \Vert_{L^{\infty}} ^{1-\frac{2}{p}}\nonumber \\ &\le& \Vert \nabla \theta_\mu(t) \Vert_{L^{2}}+\Vert \nabla \theta_\mu(t) \Vert_{L^{\infty}}.
\end{eqnarray}
Integrating in time over $[0,t]$ and employ the Cauchy-Schwartz inequality with respect to time, one gets 
\begin{eqnarray*}
\Vert \nabla \theta_\mu \Vert_{L^1_tL^{p}} \le  t^{\frac{1}{2}} \Vert \nabla \theta_\mu \Vert_{L^2_tL^{2}} +\Vert \nabla \theta_\mu \Vert_{L^1_tL^{\infty}},
\end{eqnarray*}
exploring \eqref{grad teta L^2} and \eqref{gred teta L infine}, it holds   
\begin{eqnarray*}
\Vert \nabla \theta_\mu \Vert_{L^1_tL^{p}} \le  C_0(1+t)^{7},
\end{eqnarray*}
combined with \eqref{Eq-MZ0} and \eqref{Eq-MZ} we infer that  
\begin{eqnarray*}
\Vert \theta_\mu  \Vert_{L ^1_t {B_{{p,\infty}}^{\frac{1}{p}+1}}}  \leq  C_0(1+t)^{7}.  
\end{eqnarray*}
 Inserting this estimate in \eqref{w I1} we readily get
\begin{equation}\label{grad T B}
\Vert \nabla \theta_\mu  \Vert_{L ^1_t {B_{{p,\infty}}^\frac{1}{p}}}  \leq C_0(1+t)^{7}. 
\end{equation}
Plugging \eqref{grad T B} in \eqref{w I-1} we deduce that   
\begin{equation}\label{w I*1}
 \mathrm{IV}_1  \leq C  e^{CV_\mu(t)}  \mu t^{\frac{1}{2}+\frac{1}{2p}} \Big( \Vert  \omega_\mu ^0 \Vert_{ {B_{{p,\infty}}^\frac{1}{p}}} + C_0(1+t)^{7} 
 \Big)^{\frac{1}{2}+\frac{1}{2p}}
\end{equation} 
Let us move to bound the term $\mathrm{IV}_2$. By exploiting Proposition \ref{prop2.10} and \eqref{grad T B}, on account $L ^1_t {B_{{p,\infty}}^\frac{1}{p}}=\widetilde{L} ^1_t {B_{{p,\infty}}^\frac{1}{p}}$ we find that   
\begin{eqnarray}\label{w I*2}
  \mathrm{IV}_2 &\leq & C e^{CV_\mu (t)} \mu^{\frac{1}{2p}-\frac{1}{2}} (1+\mu t)^{\frac{1}{2}-\frac{1}{2p}}   \Big( \Vert  \omega_\mu ^0 \Vert_{ {B_{{p,\infty}}^\frac{1}{p}}} +  \Vert \nabla \theta_\mu  \Vert_{\widetilde{L} ^1_t {B_{{p,\infty}}^\frac{1}{p}}}
 \Big)^{\frac{1}{2}-\frac{1}{2p}} \nonumber\\  &\leq & C e^{CV_\mu (t)} \mu^{\frac{1}{2p}-\frac{1}{2}} (1+\mu t)^{\frac{1}{2}-\frac{1}{2p}}   \Big( \Vert  \omega_\mu ^0 \Vert_{ {B_{{p,\infty}}^\frac{1}{p}}} +  C_0(1+t)^{7} \Big)^{\frac{1}{2}-\frac{1}{2p}}
\end{eqnarray}
Substituting \eqref{w I*1} and \eqref{w I*2} into \eqref{Delta v} to conclude that
\begin{equation*}
\mu \Vert \Delta v_\mu \Vert_{ L^1_tL^p} \lesssim  C_0(1+t)^{7}e^{CV_\mu (t)} (\mu t)^{\frac{1}{2}+\frac{1}{2p}} (1+\mu t)^{\frac{1}{2}-\frac{1}{2p}},
\end{equation*} 
together with \eqref{gamma}, hence Theorem \ref{Theo4.2} is proved.
\end{proof}
\subsection*{Proof of Theorem \ref{Theo1.2}} We distinguish two cases.\\ 
{\it First case: $p\in [2,+\infty[$}.  From Theorem \ref{Theo4.2}, we recall that 
\begin{eqnarray}\label{Eq-MZ4}
 \Pi(t) &\leq & Ce^{t+ V_\mu (t)+V(t)+\Vert \nabla \theta  {\Vert}_{ L^1_tL^\infty} +\Vert \nabla \theta  {\Vert}_{ L^2_tL^\infty} (\Vert \nabla \theta  {\Vert}_{L^2_t L^p}+ \Vert \nabla \theta _\mu {\Vert}_{L^2_t L^p}) } \Big( \Pi(0)\\
\nonumber&&+(1+t)^{8}(\mu t)^{\frac{1}{2}+\frac{1}{p}}(1+\mu t)  \Big).
\end{eqnarray}
To close our claim, we must estimate the two terms $\Vert \nabla \theta  {\Vert}_{L^2_t L^p}$ and $\Vert \nabla \theta  {\Vert}_{ L^2_tL^\infty}$. For this aim, Gagliardo-Nirenberg's and Young's inequlities tell us 
\begin{equation}
\Vert \nabla \theta  {\Vert}_{ L^2_tL^\infty}\lesssim \Vert \nabla \theta  {\Vert}_{ L^2_tL^2}+\Vert \nabla^2 \theta  {\Vert}_{ L^2_t L^p}. 
\end{equation}
From \eqref{grad teta L^2}  and  \eqref{grad 2 theta}, it happens  
\begin{eqnarray}\label{Eq-MZ3}
\Vert \nabla \theta_\mu \Vert_{L^{2}_tL^\infty}  &\le&  \kappa_0 \Vert  \theta^0_\mu \Vert^2_{ L^{2}} +   C_0(1+t)^{7}\\\nonumber &\le&  C_0(1+t)^{7}  .
\end{eqnarray}
Concerning $\Vert \nabla \theta  {\Vert}_{L^2_t L^p}$, again Gagliardo-Nirenberg's and Young's inequlities give in view of \eqref{grad teta L^2} and \eqref{Eq-MZ3}the following
\begin{eqnarray*}
\Vert \nabla \theta_\mu \Vert_{L^{2}_t L^p} &\lesssim &  \Vert \nabla  \theta_\mu  \Vert_{L^{2}_t L^{2}}+  \Vert \nabla  \theta_\mu \Vert_{L^{2}_t L^{\infty}}   \\ & \le &  k_0t^{\frac{1}{2}}\Vert \theta_\mu^{0}  \Vert_{L^{2}_t L^{2}}+  \Vert \nabla  \theta_\mu  \Vert_{L^{2}_t L^{\infty}} \\&\le&  C_0(1+t)^{7} .
\end{eqnarray*}
Collecting the last two estimates, hence \eqref{Lip v} and \eqref{gred teta L infine} leading to 
\begin{equation}\label{proof limit v pour p fine}
\Vert v_\mu(t) -v (t)\Vert_{L^p} + \Vert \theta_\mu (t)-\theta(t) \Vert_{L^p } \leq C_0e^{\exp{C_0t^{8}}} (\mu t)^{\frac{1}{2}+\frac{1}{2p}}. 
\end{equation}  
{\it Second case for $p=\infty$}. Gagliardo-Nirenberg's inequality yields 
\begin{eqnarray}
\Vert v_\mu(t) -v (t)\Vert_{L^\infty} + \Vert \theta_\mu (t)-\theta(t) \Vert_{L^\infty } &\leq &\Vert v_\mu(t) -v (t)\Vert_{L^2}^{\frac{1}{2}}  \Vert \nabla v_\mu(t) -\nabla v (t)\Vert_{L^\infty}^{\frac{1}{2}}  \nonumber \\&+& \Vert \theta_\mu (t)-\theta(t) \Vert_{L^2}^{\frac{1}{2}}  \Vert \nabla \theta_\mu (t)-\nabla \theta(t) \Vert_{L^\infty }^{\frac{1}{2}}. \nonumber 
\end{eqnarray}
In particular, we explore \eqref{proof limit v pour p fine} for $p=2$ and \eqref{G-th}, then it follows
\begin{equation}\label{inviscid limit for v in} 
 \Vert v_\mu(t) -v (t)\Vert_{L^\infty} + \Vert \theta_\mu (t)-\theta(t) \Vert_{L^\infty }  \leq C_0e^{\exp C_0 t^{8}} (\mu t)^{\frac{1}{4}}.
\end{equation}
On the other hand, we have $\omega_\mu -\omega = \rot( v_\mu -v )$. Taking the $L^p-$norm for this latter and using the Bernstein inequality, one obtains
\begin{eqnarray}\label{4.11}
\Vert \omega_\mu (t)-\omega(t)  \Vert_{L^p} &\leq &  \Vert  \nabla( v_\mu(t)-v(t) ) \Vert_{L^p} \nonumber \\ &\leq &  C \Vert  v_\mu(t)-v(t)  \Vert_{B^1_{p,1}}
\end{eqnarray}
Now, let $N$ be an integer which will be chosen later. Again Bernstein's inequality gives 
\begin{eqnarray}\label{4.12}
\nonumber\Vert  v_\mu(t)-v(t)  \Vert_{B^1_{p,1}} &\lesssim &  \sum_{q \leq N} 2^{q}\Vert \Delta_q(v_\mu(t)-v(t))  \Vert_{L^p} +\sum_{q > N} 2^{\frac{-q}{p}} 2^{\frac{q}{p}}\Vert \Delta_q\nabla(v_\mu(t)-v(t)) \Vert_{L^{p}}\\  
&\lesssim &  2^N\Vert   v_\mu(t)-v(t)  \Vert_{L^p} + 2^{\frac{-N}{p}}\sup_{q\ge-1}2^{\frac{q}{p}} \Vert \Delta_q\nabla(  v_\mu(t)-v(t) ) \Vert_{L^p}\\
&\lesssim& 2^N\Vert   v_\mu(t)-v(t)  \Vert_{L^p} + 2^{\frac{-N}{p}} \Vert \nabla(  v_\mu(t)-v(t) ) \Vert_{B^{\frac1p}_{p,\infty}}.\nonumber
\end{eqnarray}
Choosing $N$ be such that
\begin{equation*}
2^{N(1+\frac{1}{p})} \approx   \frac{\Vert \nabla(  v_\mu(t)-v(t) ) \Vert_{B^{\frac{1}{p}}_{p,\infty}}}{\Vert   v_\mu(t)-v(t)  \Vert_{L^p} },
\end{equation*}
combined with Calder\'on-Zugmund estimate, \eqref{4.11} and \eqref{4.12} we write
\begin{equation*}
\Vert \omega_\mu (t)-\omega(t)  \Vert_{L^p}   \lesssim  \Vert  v_\mu(t)-v(t)  \Vert_{L^p} ^{\frac{1}{p+1}} \Vert \omega_\mu(t)-\omega(t)  \Vert_{B^{\frac{1}{p}}_{p,\infty}}.
\end{equation*}
Consequently,  {\bf(1)-}Theorem \ref{Theo1.2} implies 
\begin{equation*}
\Vert \omega_\mu (t)-\omega(t)  \Vert_{L^p} \leq  C_0e^{\exp{C_0t^{8}}}  (\mu t)^\frac{1}{2p}
(1+\mu t) \Vert  \omega(t)\Vert \omega_\mu(t)-\omega(t)  \Vert_{B^{\frac{1}{p}}_{p,\infty}}. 
\end{equation*}
For the term $ \Vert \omega(t) \Vert_{B^{\frac{1}{p}}_{p,\infty}} $ already esteemed in \eqref{w I*1}  then we conclude  that
$$ \Vert \omega_\mu (t)-\omega(t)  \Vert_{L^p} \leq  C_0e^{\exp{C_0t^{8}}} (\mu t)^\frac{1}{2p}$$
To finalize, let us estimate $\|\omega_{\mu}(t)-\omega(t)\|_{B^{\frac1p}_{p,\infty}}$. To do this, using the persistence of Besov spaces explicitly formulated in the Proposition \ref{prop2.9}, one gets
\begin{eqnarray*}
\|\omega_{\mu}(t)-\omega(t)\|_{B^{\frac1p}_{p,\infty}}&\le& \|\omega_{\mu}(t)\|_{B^{\frac1p}_{p,\infty}}+\|\omega(t)\|_{B^{\frac1p}_{p,\infty}}\\
\nonumber &\le& Ce^{C(V_{\mu}(t)+V(t))}\Big(\|\omega_{\mu}^{0}\|_{B^{\frac1p}_{p,\infty}}+\|\omega^{0}\|_{B^{\frac1p}_{p,\infty}}+\|\nabla\theta_{\mu}\|_{L^1_t B^{\frac1p}_{p,\infty}}+\|\nabla\theta\|_{L^1_t B^{\frac1p}_{p,\infty}}\Big).
\end{eqnarray*}
The last two terms of the right-hand side stem from similar arguments as in \eqref{grad T B}.\\
{\bf(3)} By definition of $\Psi_\mu$ and $\Psi$ we write
\begin{equation*}
\Psi_\mu(t,x) -\Psi(t,x)=\int_0^t v_\mu(\tau,\Psi_\mu(\tau,x)) -v(\tau,\Psi_\mu(\tau,x))d\tau.
\end{equation*}
Consequently, 
\begin{eqnarray*}
|\Psi_\mu(t,x) -\Psi(t,x)|  & \leq  & \int_0^t |v_\mu(\tau,\Psi_\mu(\tau,x)) -v(\tau,\Psi_\mu(\tau,x)) |d\tau \nonumber \\ &+&\int_0^t |v(\tau,\Psi_\mu(\tau,x)) -v(\tau,\Psi_\mu(\tau,x)) |d\tau \nonumber.
\end{eqnarray*}
The first term of r.h.s. follows from \eqref{inviscid limit for v in}, that is
\begin{equation}\label{flow-1}
\int_0^t |v_\mu(\tau,\Psi_\mu(\tau,x)) -v(\tau,\Psi_\mu(\tau,x)) |d\tau\le C_0e^{\exp C_0 t^{8}} (\mu t)^{\frac{1}{4}}. 
\end{equation}
For the second term, exploring in general case the following relation  
\begin{eqnarray*}
|f\circ\Psi_{\mu}-f\circ\Psi|&=&\frac{|f\circ\Psi_{\mu}-f\circ\Psi|}{|\Psi_{\mu}-\Psi|}|\Psi_{\mu}-\Psi|\\
&\le&\|\nabla f\|_{L^\infty}\|\Psi_{\mu}-\Psi\|_{L^\infty}.
\end{eqnarray*}
one may deduce that  
\begin{equation}\label{flow-2}
\int_0^t |v(\tau,\Psi_\mu(\tau,x)) -v(\tau,\Psi(\tau,x)) |d\tau\le \int_0^t\|\nabla v(\tau)\|_{L^\infty}\|\Psi_{\mu}(\tau)-\Psi(\tau)\|_{L^\infty}d\tau. 
\end{equation}
Gathering \eqref{flow-1} and \eqref{flow-2}, then it follows
\begin{equation*}
|\Psi_\mu(t,x) -\Psi(t,x)|\le C_0e^{\exp C_0 t^{8}} (\mu t)^{\frac{1}{4}}+\int_0^t\|\nabla v(\tau)\|_{L^\infty}\|\Psi_{\mu}(\tau)-\Psi(\tau)\|_{L^\infty}d\tau. 
\end{equation*}
Gronwall's inequality yields 
\begin{equation*}
\Vert \Psi_\mu(t) -\Psi(t)\Vert_{L^\infty}   \leq  C_0e^{\exp{C_0t^{8}}} (\mu t)^{\frac{1}{4}}. 
\end{equation*}
This completes the proof of  Theorem \ref{Theo1.2}.
\subsection*{Acknowledgement} The second author would thank Professor Changxing Miao for his comments and serious several suggestions.

\end{document}